\documentclass[11pt]{article}
\usepackage{amsmath, amsfonts, amssymb, amsthm, amstext, latexsym, hyperref, comment,mathtools, mathrsfs}
\usepackage[margin=1in]{geometry}
\usepackage{enumerate,yfonts,eucal}
\usepackage{authblk}
\usepackage[dvipsnames]{xcolor}
\usepackage{csquotes}
\usepackage[symbol]{footmisc}

\usepackage[color, outerbars]{changebar}
\cbcolor{red}

\usepackage[
    backend=bibtex, 
    maxnames=10,
    sortcites=true,
    doi=false,url=false,
    giveninits=true, hyperref]
    {biblatex}
    \renewbibmacro{in:}{}
\newtheorem{thm}{Theorem}

\newtheorem{lem}{Lemma}[section]
\newtheorem{rmk}{Remark}
\newtheorem*{rmk*}{Remark}
\newtheorem{prop}[lem]{Proposition}
\newtheorem{clm}[lem]{Claim}

\newtheorem{obs}[lem]{Observation}
\newtheorem*{obs*}{Observation}

\newtheorem{cor}[lem]{Corollary}
\newtheorem*{cor*}{Corollary}

\makeatletter
\newcommand{\vast}{\bBigg@{4}}
\newcommand{\Vast}{\bBigg@{5}}
\makeatother

\newcommand{\norm}[1]{\left\Vert#1\right\Vert}
\newcommand{\wh}{\widehat}
\newcommand{\wt}{\widetilde}
\newcommand{\sinc}{\mathrm{sinc}}

\newcommand{\per}[1]{\mathcal{P}_{#1}^\ell (T)}
\newcommand{\perl}[2]{\mathcal{P}_{#1}^{#2} (T)}
\newcommand{\perlt}[3]{\mathcal{P}_{#1}^{#2} (#3)}

\newcommand{\ttt}[1]{\theta_{#1}^\ell (T)}
\newcommand{\tttl}[2]{\theta_{#1}^{#2} (T)}
\newcommand{\tttlt}[3]{\theta_{#1}^{#2} \left(#3\right)}
\DeclareMathOperator{\TV}{TV}
\DeclareMathOperator{\ac}{ac}
\DeclareMathOperator{\sing}{sing}

\makeatletter
\newcommand*\rel@kern[1]{\kern#1\dimexpr\macc@kerna}
\newcommand*\widebar[1]{%
  \begingroup
  \def\mathaccent##1##2{%
    \rel@kern{0.8}%
    \overline{\rel@kern{-0.8}\macc@nucleus\rel@kern{0.2}}%
    \rel@kern{-0.2}%
  }%
  \macc@depth\@ne
  \let\math@bgroup\@empty \let\math@egroup\macc@set@skewchar
  \mathsurround\z@ \frozen@everymath{\mathgroup\macc@group\relax}%
  \macc@set@skewchar\relax
  \let\mathaccentV\macc@nested@a
  \macc@nested@a\relax111{#1}%
  \endgroup
}
\makeatother

\newcommand{\alignedintertext}[1]{%
  \noalign{%
    \vskip\belowdisplayshortskip
    \vtop{\hsize=\linewidth#1\par
    \expandafter}%
    \expandafter\prevdepth\the\prevdepth
  }%
}

\newcommand{\FF}{{\mathcal F}}
\newcommand{\cF}{{\mathcal F}}
\newcommand{\R}{\mathbb R}

\newcommand{\E}{\mathbb E}
\newcommand{\N}{\mathbb N}
\newcommand{\Z}{\mathbb Z}
\newcommand{\lm}{\lambda}
\newcommand{\si}{\sigma}
\newcommand{\g}{\gamma}
\newcommand{\al}{\alpha}
\newcommand{\ep}{\varepsilon}
\newcommand{\p}{\varphi}
\newcommand{\alt}
{\alpha'}
\newcommand{\Alt}{D}

\renewcommand{\P}{\mathbb{P}}
\newcommand{\var}{\mathrm{var}\,}
\newcommand{\sprt}{\mathrm{sprt}\,}
\newcommand{\cov}{\mathrm{cov}\,}

\newcommand{\ind}{1{\hskip -2.5 pt}\hbox{I}}
\newcommand{\D}{\Delta}

\newcommand{\calN}{\mathcal{N}}
\newcommand{\cN}{\mathcal{N}}
\newcommand{\cL}{\mathcal L}
\newcommand{\cM}{\mathcal{M}}
\newcommand{\cB}{\mathcal{B}}
\newcommand{\cA}{\mathcal{A}}

\newcommand{\enn}{n}

\usepackage{setspace}
\definecolor{dark gray}{gray}{0.3}
\def\gray#1{{\color{dark gray} #1}}

\numberwithin{equation}{section}

\makeatletter
\newcommand{\subjclass}[2][1991]{%
  \let\@oldtitle\@title%
  \gdef\@title{\@oldtitle\footnotetext{#1 \emph{Mathematics subject classification.} #2}}%
}
\newcommand{\keywords}[1]{%
  \let\@@oldtitle\@title%
  \gdef\@title{\@@oldtitle\footnotetext{\emph{Key words and phrases.} #1.}}%
}
\makeatother

\title{
Persistence and Ball Exponents for Gaussian Stationary Processes
}
\author{Naomi D. Feldheim\thanks{Department of Mathematics, Bar-Ilan University, Israel.  \texttt{naomi.feldheim@biu.ac.il}. \\Supported in part by ISF grant 1327/19.}, \hspace{2pt}
Ohad N. Feldheim\thanks{Einstein Institute for Mathematics, Hebrew University of Jerusalem, Israel. \texttt{ohad.feldheim@mail.huji.ac.il}. \\ 
Supported in part by ISF grant 1327/19.} \hspace{2pt}
and Sumit Mukherjee\thanks{Department of Statistics, Columbia University, New-York, USA. \texttt{sm3949@columbia.edu}. 
\\ Supported in part by NSF grants DMS-1712037 and DMS-2113414.
}}

\begin{document}
\subjclass[2000]{60G15, 60G10, 42A38}
\keywords{Gaussian process, stationary process, spectral measure, persistence, ball probability, small deviations, small ball, one-sided barrier}

\date{}
\maketitle

\begin{abstract}
Consider a real Gaussian stationary process $f_\rho$, indexed on either $\R$ or $\Z$ and admitting a spectral measure $\rho$. We study $\theta_{\rho}^\ell=-\lim\limits_{T\to\infty}\frac{1}{T} \log\P\left(\inf_{t\in[0,T]}f_{\rho}(t)>\ell\right)$, the persistence exponent of $f_\rho$.
We show that, if $\rho$ has a positive density at the origin, then the persistence exponent exists; moreover, if $\rho$ has an absolutely continuous component,
then $\theta_{\rho}^\ell>0$ if and only if this spectral density at the origin is finite. We further establish continuity of $\theta_{\rho}^\ell$ in $\ell$, in $\rho$ (under a suitable metric) and, if $\rho$ is compactly supported, also in dense sampling. 
Analogous continuity properties are shown for  $\psi_{\rho}^\ell=-\lim\limits_{T\to\infty}\frac{1}{T} \log\P\left(\inf_{t\in[0,T]}|f_{\rho}(t)|\le \ell\right)$, the ball exponent of $f_\rho$, and it is shown to be positive if and only if $\rho$ has an absolutely continuous component.
\end{abstract}

\section{Introduction}

\emph{Persistence}, namely the event that a real stochastic process $f$ stays above a level $\ell$ over a long time interval $[0,T]$, is a well studied object since the 1950s, especially for centered \emph{Gaussian stationary processes} (GSP) in discrete or continuous time (c.f. \cite{Rice45,slepian62,NR62,Shur1965,Pickands69,Borell75,CIS76,ST05,Molchan12,AS} and the references therein).
This has been studied in particular for the critical case $\ell=0$ (c.f. \cite{Sakagawa15, FFN,KK16,FFJNN18}),
with applications to statistical mechanics (c.f. \cite{BMS, MBE01,MBE02,Watson96,PS18, KST}).
The persistence exponent of a GSP $f$ over level $\ell$ is defined as the exponential rate of decay of the persistence probability, namely,
$$\theta^\ell_f=-\lim_{T\to\infty}\frac{1}{T}\log\P\left(\inf_{t\in[0,T]}f(t)>\ell\right),$$
whenever the limit exists.

Slepian, in his celebrated 1962 paper \cite{slepian62}, conjectured that a persistence exponent
should exist under mild conditions.
The validity of this conjecture has often been taken for granted in the physics literature \cite{DHZ, EB}.
Prior to this work, the only cases in which an exponent was shown to exist were non-negatively correlated processes (using a subadditivity argument,  {see for example \cite{LiShao,DM2}}), Markov processes  (using Perron-Frobenius), and $m$-dependent processes (using independence properties) \cite{AMZ}.

Every continuous GSP $f:\R\to\R$ is characterized by a \emph{covariance kernel} $r(t)=\cov(f(0),f(t))$, or, equivalently, by a \emph{spectral measure}, that is, a finite, non-negative, symmetric measure $\rho$ on $\R$ such that:
\[
r(t)=\wh{\rho}(t)  = \int_\R e^{-it \lm} \ d\rho(\lm).
\]
We denote by $f_\rho, r_\rho, \theta^\ell_\rho$ the process, covariance kernel and persistence exponents associated with $\rho$.

In recent years it became evident that it is possible to obtain a much more precise understanding of persistence in terms of the spectral measure of $f$ (c.f. \cite{FF,FFN,FFJNN18}).
These spectral methods have already found applications to homogeneous Gaussian fields~\cite{Molchan22}.

 In this paper we show that 
 the
existence of a positive spectral density at the origin is a sufficient (and nearly necessary) condition for the existence of $\theta^\ell_\rho$. We further provide several new continuity results. We do so by combining and expanding the spectral method of \cite{FFN} and the covariance method of \cite{DM2}.


\subsection{Main results}

Throughout, we require $\rho$ to belong to the following class of measures with finite $\log$--$(1+\beta)$ moment:
\[
\cL =\Big\{\rho\in \mathcal{S} \ \mid  \exists \beta>0: \: \int_0^\infty \max\Big(\log^{1+\beta}\lambda,1\Big)d\rho(\lambda)<\infty \Big\},
\]
where $\mathcal{S}$ is the set of all spectral measures (symmetric, non-negative, finite measures on $\R$).
The requirement $\rho\in\cL$ is a rather mild condition which ensures continuity of the process (see~\cite[Sec.~1]{AT}).
While we state our results for GSPs on $\R$, all theorems are valid for processes on $\Z$ as well, in which case $\rho$ is supported on $[-\pi,\pi]$ and the condition $\rho\in \cL$ is always satisfied (see Remark~\ref{rmk: Z}).
Define 
\[
\rho'(0):=\lim_{\ep\to 0}\tfrac1{2\ep}\rho([-\ep,\ep])
\]
whenever the limit exists, and 
\[
 \cM =\Big\{\rho\in \mathcal{S} \ \mid  \rho'(0) \in (0,\infty] \Big\}.
\]
We use the notation $\rho_{ac}$ to denote the absolutely-continuous component of the measure $\rho$.
\begin{thm}\label{thm: main exist}
Let $\rho\in \cL\cap \cM$. Then, for all $\ell\in\R$ a persistence exponent
$\theta_\rho^\ell\in[0,\infty)$ exists.
Moreover,
under the further assumption $\rho_{ac}\ne 0$, we have
$\theta_\rho^\ell>0$ if and only if $\rho'(0)<\infty$.
\end{thm}

\noindent

\begin{rmk}\label{rmk: cexm}
{\rm
Tightness of the condition $\rho \in \cM$ is demonstrated by a counter-example, provided in Section~\ref{sec: exist counter}. There we show that
for the spectral density
$(A + B \cos(\tfrac 1 \lm))\textrm{\ind}_{|\lm|\le 1}$ in a certain range of $0<B<A$, the exponent $\theta_\rho^0$ does not exist.
In this example the density is bounded, compactly supported, and continuous on $[-1,1]\setminus\{0\}$.
We note that it is possible to construct examples of non-existence of $\theta_\rho^\ell$ for any level $\ell\ge 0$, and even for all levels at once.
}
\end{rmk}

\begin{rmk}\label{rmk: inf}{\rm
For $\rho\in \cL$ with $\rho_{\ac}\neq 0$ and $\rho'(0)=0$, we have $\theta_{\rho}^\ell=\infty$ for all $\ell\ge 0$. 
This is not covered by Theorem~\ref{thm: main exist}, but follows from Lemma~\ref{lem: FF UB} below. 
The conditions for positivity of $\theta_\rho^\ell$ in the case $\rho_{\ac}=0$ presently remain unknown.
}
\end{rmk}
\begin{rmk}\label{rmk: singular}{\rm
For $\ell\le 0$ and $\rho\in \cM$, it follows from our results that $\theta_\rho^\ell$ is positive
if and only if $\rho_{\ac}\neq 0$ . For $\ell> 0$ and $\rho\in \cM$ we conjecture that $\theta_\rho^\ell$ is always positive.}
\end{rmk}
\bigskip
A \emph{ball event} is the event that a real stochastic process $f$ stays in $[-\ell, \ell]$ over a long time interval $[0,T]$. This event is well-studied (c.f. \cite{Li99, AIL09,HLN16,Lifshits87,LiShao01}, for a complete bibliography, see \cite{Lifshits10}) with various applications (c.f. \cite{Lifshits13}).
Given a spectral measure $\rho$, denote the $\ell$-ball exponent of $\rho$ by
$$ \psi_{\rho}^\ell=-  \lim_{T\to\infty} \frac{1}{T}\log\P\left(\sup_{t\in[0,T]}|f_{\rho}(t)|<\ell\right).$$
This exponent is known to exist for all spectral measures in $\cL$ using a subadditivity argument together with Khatri-Sidak \cite{Sidak68} inequality (see e.g. \cite{Li99}).
Our main result concerning ball events is the following analogue of Theorem~\ref{thm: main exist}.

\begin{thm}\label{thm: nontrivial ball}
Let $\rho\in \cL$. For all $\ell$ the ball exponent $\psi^\ell_\rho$ is positive
if and only if $\rho_{\ac}\neq 0$.
\end{thm}

We further obtain three results concerning with the continuity and comparison of persistence and ball exponents.
The first is the fact that
the persistence and ball exponents are monotone in the spectral measure (in an appropriate sense), and are unaffected by the singular component.

\begin{thm}\label{thm: singular}
Let $\ell\in\R$ and $\rho, \nu\in \cL$. Then:
\begin{enumerate}[{\rm (I)}]
\item\label{item: singular ball}
$\psi_{\rho+\nu}^\ell \ge \psi_\rho^\ell$.
\item \label{item: singular pers}
$ \theta^\ell_{\rho + \nu} \ge \theta^\ell_\rho $, provided that $\rho\in \cM$ and
$\nu'(0)=0$.
\end{enumerate}
Further, equality holds in both \eqref{item: singular ball} and \eqref{item: singular pers} if 
 $\nu$ is purely-singular.
\end{thm}
\noindent Observe that Theorem~\ref{thm: nontrivial ball} is an immediate corollary of Theorem~\ref{thm: singular}.

\begin{rmk}\label{rmk: iff sing}
{\rm
We conjecture that equality in Theorem~\ref{thm: singular} holds if and only if $\nu$ is purely-singular.
}
\end{rmk}

The second result is
continuity of exponents in the level $\ell$ and in the measure $\rho$, with respect to a suitable metric topology.
Stating the result requires the following definitions.

For two finite measures $\rho_1, \rho_2$, define the total variation distance
\[
d_{TV}(\rho_1,\rho_2)=\sup\{|\rho_1(E)-\rho_2(E)|: \: E \text{ is a 
 measurable subset of }\R \}.
\]
If $\rho_1, \rho_2$ also have a finite density at the origin, we define the metric
\begin{equation}\label{eq: metric}
d_{\TV_0}(\rho_1,\rho_2)=
d_{TV}(\rho_1,\rho_2)+|\rho_1'(0)-\rho_2'(0)|.
\end{equation}
For  {$\al\ge 0$ and} $\al',A,\beta,B>0$,  consider the classes of measures
\begin{align}
\begin{split}\label{eq: classes}
\cM_{(\al,\alt),A} &=\Big\{\rho\in \mathcal{S} \ \mid  \al\le \frac{\rho(-x,x)}{2x} \le  \alt, \forall x\in (0,A) \Big\},\\
\cL_{\beta,B} &=\Big\{\rho\in \mathcal{S} \ \mid  \int_0^\infty \max\Big(\log^{1+\beta}\lambda,1\Big)d\rho(\lambda)<B \Big\}.
\end{split}
\end{align}

\begin{thm}\label{thm: cont}
Fix $\al,\al',A,\beta,B>0$.
\begin{enumerate}[{\rm (I)}]
\item\label{item: cont ball}
The ball exponent $\psi_\rho^\ell$ is locally-Lipschitz continuous in $\ell\in (0,\infty)$ and uniformly
$d_{\TV}$-continuous in $\rho\in \cL_{\beta,B}$.
\item \label{item: cont pers}
The persistence exponent $\theta_\rho^\ell$ is locally-Lipschitz continuous in $\ell\in\R$ and uniformly $d_{\TV_0}$-continuous in
$\rho \in \cL_{\beta,B} \cap \cM_{(\al,\al'),A}\cap \cM$.
\end{enumerate}
\end{thm}

\begin{rmk}{\rm
The function $\ell \mapsto \psi_\rho^\ell$ is finite and convex (by log-concavity of centered Gaussian processes), and hence continuous.
In contrast, continuity of $\ell\mapsto \theta_\rho^\ell$  {does not hold} for all spectral measures $\rho$.  {To see this,}
consider, for example, $\mu = \frac 1 2 (\delta_{1} + \delta_{-1})$, which corresponds to the process
$f_\mu (t) =  \zeta_1 \cos(t) + \zeta_2 \sin (t)$
where $\zeta_1,\zeta_2\sim \cN(0,1)$ are independent.
A direct analysis yields
\[ \theta_\mu^\ell =
 \begin{cases}
\infty, &\ell \ge 0,\\
0, &\ell <0.
\end{cases}
\]

Thus, the assumption $\rho'(0)>0$ is necessary to ensure continuity of the map $\ell\mapsto\theta_\rho^\ell$ for all $\ell$. 
Nevertheless,
whenever 
$\theta_\rho^\ell$ exists, it is always continuous 
for $\ell<0$ (even without this assumption).
On the other hand, if $\rho'(0)=0$ and  $\rho_{\ac}\neq 0$, then $\theta^\ell_\rho=\infty$ for all $\ell\ge 0$ (by Remark~\ref{rmk: inf}).
Therefore whenever  $\rho_{\ac}\neq 0$, discontinuity of the map $\ell\mapsto \theta_\rho^\ell$ can occur only at $0$. We conjecture that this is the case also for purely singular measures.

}\end{rmk}

\medskip
The third continuity result is concerned with discrete sampling.
For a spectral measure $\rho$ define
\begin{align*}
\theta^\ell_{\rho;\D}=&-\lim_{T\to\infty}\frac{1}{T}\log\P\left(\inf_{n\in \Z, \, n\D \in [0,T]}f_\rho(n\D)>\ell\right),\\
\psi^\ell_{\rho; \D}=&-\lim_{T\to\infty}\frac{1}{T}\log\P\left(\sup_{n\in \Z, \,n\D \in [0,T]} \left|f_\rho(n\D) \right| < \ell\right),
\end{align*}
whenever the limit exists. These are the persistence and ball exponents of the Gaussian stationary sequence generated by sampling $f$ in $\D$-intervals.
Our next result provides conditions for convergence of $\theta_{\rho;\D}^\ell$ to $\theta_{\rho}^\ell$ and of $\psi_{\rho;\D}^\ell$ to $\psi_\rho^\ell$, as the sampling interval $\D$ approaches $0$.
\begin{thm}\label{thm: sample}
Let $\ell\in\R$ and $\rho\in\cL$. Then:
\begin{enumerate}[{\rm (I)}]
\item\label{item: sample ball}
$
\lim\limits_{\D\to 0}\psi^\ell_{\rho;\D}= \psi^\ell_{\rho}.
$
\item\label{item: sample above}
$
\lim\limits_{\D\to 0 }\theta^\ell_{\rho;\D}= \theta^\ell_{\rho},
$
provided that $\rho\in\cM$ has compact support.
\end{enumerate}
\end{thm}

\begin{rmk}{\rm
The second part of Theorem \ref{thm: sample} may be extended to non-compactly supported spectral measures with sufficient rate of decay at infinity (for instance, ones which have density $\rho'$ satisfying $\sup_{\lm\in\R} |\lm|^{1+\eta} \rho'(\lm)<\infty$ for some $\eta>0$).
However, it does not extend to all $\rho\in \cL \cap \cM$. A counter-example is provided by Proposition~\ref{prop:non-convergence}.
}\end{rmk}

\begin{rmk}\label{rmk: Z}{\rm
As we have mentioned earlier, Theorems~\ref{thm: main exist}, \ref{thm: nontrivial ball}, \ref{thm: singular} and \ref{thm: cont} 
hold true also for Gaussian processes in discrete time. The proofs remain valid without any change.
For such processes the condition $\rho \in \cL$ holds trivially.  
}\end{rmk}

\begin{rmk}{\rm A natural generalization of the problem presented here is the so called \emph{two sided barrier problem} (considered by Shinozuka \cite{Shinozuka65}), i.e. the probability that a GSP persists within a set $[a,b]$ where $a \ne -b$ and $-\infty < a < b<\infty$. Somewhat surprisingly the methods used here do not seem to generalize directly to this case as several monotonicity properties are lost. Thus the problem remains open. It is possible to show that when $a<0<b$, an exponential-type behavior is always demonstrated, while if  $0<a<b$ such exponential-type behavior holds only when the spectral measure is well behaved about the origin. We conjecture that, in both cases, the existence of a two-sided barrier exponent should hold whenever $\rho\in\cL\cap \cM$.
}\end{rmk}

\subsection{Background}

\subsubsection{Gaussian stationary processes and persistence}
Gaussian processes provide good approximations for natural phenomena in which a random function is generated as a sum of nearly independent random contributions of similar scale (due to the functional CLT).
When the process has a time invariant distribution, stationarity occurs and the approximating process becomes a GSP. This makes GSP an excellent model for noise, such as static interference, liquid surface fluctuations, gas density fluctuations or the shot effect fluctuation in thermionic emission.
GSPs have therefore been extensively studied with motivation stemming from mathematics, physics and engineering. For an introduction to Gaussian processes see~\cite{AT, Lifshits13}.

In 1944, Rice~\cite{Rice45} studied the zeroes of GSPs, introduced the notion of persistence and presented the asymptotics of the probability of persistence in short intervals $[0,T]$ as $T$ tends to $0$. In the same paper the problem of estimating  persistence probability decay as $T$ tends to $\infty$ was first posed.
This was also the topic of a series of experimental and heuristic papers by Kuznetsov, Stratonovich and Tikhonov in the 1950's (translated and collected in~\cite[Ch. IV]{KST}).

Motivated by this problem, Slepian~\cite{slepian62} introduced in 1962 his famous inequality, and estimated the persistence probability of several examples. He conjectured that under mild decay conditions, the persistence should decay exponentially. In his words:
\begin{displayquote}
``Intuition would indicate exponential falloff for a wide class of covariances.''
\end{displayquote}
Slepian also called for a study of continuity properties of the persistence exponent (in terms of the covariance kernel), and for numerical methods for estimating it.

Newell and Rosenblatt \cite{NR62} were quick to extend Slepian's work and apply his methods to obtain rough bounds for the persistence probability for GSPs with polynomially decaying covariance. These were far from being tight, but remained the state-of-the-art for at least fifty years.

In the 1990's the interest of the physics community in persistence revived,
as it turned out to be of use for analyzing rare events in spin systems and heat flows (see e.g.~\cite{DHZ,MBE01} and the extensive survey~\cite{BMS}).
Indeed, the authors of \cite{DHZ} were somewhat disappointed at the state of the problem:
\begin{displayquote}
``To our surprise, given the correlation function of the Gaussian process the determination of this asymptotic decay turns out to be a hard unsolved problem.''
\end{displayquote}

In the late 2010's, Dembo and the third author \cite{DM, DM2}, seeking to study both the solutions of the heat equation initiated by white noise and the probability that a random polynomial has no roots, revisited Slepian's method. They observed that in the restricted case of GSPs with non-negative covariance, it is possible to use probabilistic arguments and tools from linear algebra to extend the method and obtain the exact rate of the decay of the persistence up to sub-exponential factors.
Developing on this, in ~\cite[Lem. 3.2]{AM} Aurzada and Mukherjee showed that a GSP with non-negative covariance has a positive persistence exponent if and only if its correlation is integrable. In addition, they obtain continuity results in terms of the covariance kernel for this set of processes.

The first process with sign-changing covariance kernel for which exponential-type decay was established is the $\sinc$ kernel process. This result, due to Antezana, Buckley, Marzo and Olsen~\cite{ABMO}, was a tour de force of analytic methods and direct computations.
A study of this result, has led the first two authors to introduce spectral conditions which ensure exponential bounds on persistence, requiring the spectral measure to have a polynomial decay and a bounded spectral density in a small vicinity of the origin~\cite{FF}. This was extended together with Nitzan \cite{FFN} to conditions under which persistence decays sub- or super-exponentially, providing also new examples for extremely fast decaying persistence probabilities. All of these conditions depend on the interplay between the spectral behavior near the origin and the decay of the spectral measure near infinity. The special case of a ``spectral gap'' (i.e., a spectral measure which vanishes on an interval near the origin) was treated in more detail in \cite{FFJNN18}.

\subsubsection{Related processes and events}
Persistence is sometimes regarded more generally, as the event that a given stochastic process
takes values in a specific set over a long time interval.
Such events have received significant attention for a large class of examples, including
random walks~\cite{Bingham,Feller} (see references there-in), L\'evy processes~\cite{Bertoin, Doney}, Markov processes~\cite{Darroch-Seneta, Collet-Martinez-Martin}, random polynomials~\cite{littlewood_offord,dembo2002random} and (spatial) processes driven by either stochastic differential equations, or by partial differential equations with random initial configuration~\cite{Sakagawa15,SM08}.

For point configurations in the plane, much attention is given to the persistence-type event of having no points at all in a large region. This event, which reflects the rigidity of the model, was studied for zeroes of random analytic functions \cite{BNPS, N, ST05}, Coulomb gas \cite{GN} and random matrices~\cite{BZ98}, among other examples. In some cases, even very advanced questions such as the conditional behavior of the model could be handled \cite{GN1, GP, ASZ14}.

Ball probabilities have their own long history of studies.
The topic is often referred to in the literature as ``small deviations theory'', as it analyzes events where the process is confined to a much smaller radius $\ell$ compared with the length of the time interval $T$. Various models and metrics were considered, 
with significant interest in Gaussian processes and the $L_\infty$ metric. 
In this context, a classical bound is due to Lifshits and Tsirelson~\cite{Lifshits87}, and improvements were made by Aurzada, Ibragimov, Lifshits and van Zanten \cite{AIL09} as well as by Weber~\cite{Weber89, Weber13}. 
In the current paper, we focus on the case where $\ell$ stays fixed, and $T$ goes to $\infty$.

\subsection{Comparison lemmata for persistence probabilities}\label{sec: lemmas present}

In this section we present three results which act as basic tools in our treatment of persistence probabilities, but are also of independent interest.
These are three comparison properties of persistence probabilities: under change of level, under smoothing and under change of measure.

Denote
\begin{align}\label{eq: per}
 \per{\rho} :=\P\left(\inf_{t\in [0,T]}f_\rho(t)>\ell\right), \qquad \theta_{\rho}^{\ell}(T)  := -\frac 1 T \log \perl{\rho}{\ell}.
\end{align}
and recall that $\theta_\rho^\ell = \displaystyle\lim_{T\to\infty} \theta_\rho^\ell(T)$ 
whenever the limit exists.

Throughout the remainder of the paper, we fix parameters $\ell\in \R$ and $\al, \beta, B>0$,
and consider measures in the class
\[
\cM_{\al,A} :=\bigcup\limits_{\al'>0}\cM_{(\al,\al'),A} = \Big\{\rho\in \mathcal{S} \ \mid  \al\le \frac{\rho(-x,x)}{2x}, \: \forall x\in (0,A) \Big\},
\]
in $\cL_{\beta,B}$, 
 {and, occasionally, in $\cM_{(0,\al'),A}$ for $\al'>0$ as introduced in~\eqref{eq: classes}.}
In what follows, constants may depend on $(\ell,\al,\beta,B)$ implicitly.  {In contrast, dependence on $\al', A$ will always be made explicit, as in applications these may vary with $T$.}

 {Our first lemma provides continuity in levels: it shows that $\ttt{\mu}$
is approximately Lipschitz continuous (for large $T$) 
with respect to the level $\ell$, 
uniformly over the class of measures $\mu \in\left(\cup_{A>0}\cM_{\al,A}\right) \cap \cL_{\beta,B}$.} 

\begin{lem}[continuity in levels]\label{lem: level cont}
There exists $C>0$ such that, for all $\mu\in \cM_{\al,A} \cap \cL_{\beta,B}$, $\delta>0$  {and $T\ge \max\{4, \tfrac 1 A\}$,} we have
\begin{equation*}
0\le \theta_{\mu}^{\ell+\delta}(T) - \theta_{\mu}^{\ell-\delta}(T) \le C \delta 
 {+ 
\tfrac{\log 2}{T}.}
\end{equation*} 
\end{lem}

 {Our second lemma is concerned with a \emph{smoothing} operation, 
carried out by taking a compact sliding-window weighted-average to a component of the process. The lemma states that smoothing results in a higher persistence probability, comparing with the original process on a slightly longer interval.
The formulation here is given in spectral terms; for a temporal interpretation, see  Observation~\ref{obs: emp conv} below.
}

\begin{lem}[smoothing lemma]\label{lem: convolution}
Fix $a>0$.
Let $h$ be a spectral density such that $h(0)=1$, $\wh{h}\ge 0$ and
$\wh{h}$ is supported on $[-\tfrac a 2, \tfrac a 2]$. Then for any spectral measures $\mu$ and $\nu$:
\[
 \perlt{\mu + \nu}{\ell}{T+a} \le \per{\mu + h^2 \nu},
\]
provided that $h^2$ is integrable with respect to $\nu$.
\end{lem}

 {Our last lemma establishes continuity of $\theta_\mu^\ell(T)$ with respect to perturbations of the measure $\mu$: it cannot increase much when adding to $\mu$ a measure of small total mass, and it cannot decrease much when adding a measure with small local density at the origin (comparing with a slightly shorter time interval).
}

\begin{lem}[continuity in measure]\label{lem: meas cont}
Let $\mu\in \cM_{\al,A} \cap \cL_{\beta,B}$ and  $\nu\in\cL_{\beta,B}$.

\begin{enumerate}[{\rm (I)}]

\item\label{item: cont a}
If $\nu(\R)<\ep$, 
 {then for all $T\ge\max\{4,\tfrac 1 A,\tfrac 1 \ep\}$} we have
\begin{equation*}
\ttt{\mu+\nu} \le \ttt{\mu} +  C_\ep,
\end{equation*}
where $\lim\limits_{\ep\to 0} C_\ep =0 $.

\item\label{item: cont b}
There exists 
$c=c(\al,\beta,B,\ell)>0$ such that, if $\nu\in
\cM_{(0,\ep),A}\cap\cM_{(0,\al'),\Alt}$ for some $\ep,A,D\in (0,1),\alt>0$, then for all
$T\ge \frac{c}{A\sqrt{\ep}}$
 we have 
\begin{equation*}
\tttlt{\mu}{\ell}{ T(1-c\ep^{1/8}) } \le  \ttt{\mu+\nu}  + C_{\ep,\al'},
\end{equation*}
where
$\lim\limits_{\ep\to 0} C_{\ep,\al'}=0$.

\end{enumerate}

\end{lem}

\subsection{Examples}
We conclude the introduction by discussing several noteworthy processes for which the existence of persistence exponents is first established by Theorem~\ref{thm: main exist}.

\begin{enumerate} 

\item
\textbf{Sinc kernel.}
Consider the process with kernel $r(t) = \sinc(t) =  \frac{\sin(\pi t)}{\pi t}$. The corresponding spectral measure has density $\ind_{[-\pi,\pi]}$.  
It is clear that $\rho\in \cM\cap \cL$, so by Theorem~\ref{thm: main exist}, for any $\ell\in\R$ the exponent $\theta_\rho^\ell$ exists in $(0,\infty)$. 
This process received much attention~\cite{ABMO} and has various applications to statistics and signal processing~\cite{Tobar19}.

\item \textbf{Zero-order Bessel kernel.}  
Let
$r(t)=J_0(t)=\sum_{m=0}^\infty \frac{(-t)^m}{2^m(m!)^2}.$
The corresponding spectral measure $\rho$ has a density 
$\frac{2}{\sqrt{1-\lambda^2}}\ind_{[-1,1]}$ (see~\cite{Temme}). As $\rho\in \mathcal{M}\cap \mathcal{L}$, Theorem~\ref{thm: main exist} yields that,
for any $\ell\in\R$, the exponent $\theta_\rho^\ell$ exists in $(0,\infty)$.

\item\textbf{Moving Average over i.i.d.}
For a sequence $\{U_k\}_k\in\Z$ such that $\sum_{k\in\Z} U_k^2<\infty$, we define the moving average process by
\[
X_j =  \sum_{k\in\Z} U_k Z_{j-k} = (U* Z)(j),
\]
where $\{Z_k\}_{k\in \Z}$ are i.i.d. $\cN (0,1)$ random variables. 
The corresponding spectral measure $\rho$ has density $|\wh{U}(\lm)|^2$, where 
$\wh{U}(\lm) = \sum_{k\in\Z} U_k e^{i k \lm}$. 
Assume that 
$\wh{U}$ is continuous in a neighborhood of $0$ (in the wide sense) and note that $\wh{U}(0) = \sum_{k\in\Z} U_k$. 
Then
$\rho'(0)=\lim_{x \to 0} \frac 1 x \int_{0}^{x} |\wh{U}(\lm)|^2 d\lm$ exists and equals $|\wh{U}(0)|^2$.
By Theorem~\ref{thm: main exist} and Remarks~\ref{rmk: inf} and~\ref{rmk: Z}, the persistence exponent $\theta_X^\ell$ exists in $[0,\infty]$ for any $\ell\in \R$, and
\[
\theta_X^\ell =\infty \iff \sum_{k\in\Z} U_k =0, \quad \text{while } \quad
\theta_X^\ell =0 \iff |\sum\limits_{k\in\Z} U_k| =\infty.
\]

\item\textbf{Moving Average over a GSP.}
More generally, let $\{Y_t\}_{t\in T}$ be a GSP over $T\in\{\Z,\R\}$ with spectral measure $\rho\in \cM\cap\cL$. Let $U: T\to\R$ be such that 
$\wh{U}(\lm) = \int_T U(t) e^{i t\lm} dt$ is continuous (in the wide sense) in a neighborhood of $0$, and further $\int |\wh{U}(\lm)|^2 d\rho(\lm) <\infty$. 
Then the moving-average process defined by $X(t) = (U*Y)(t)$ is a GSP (see Obs.~\ref{obs: emp conv} below), for which $\theta_X^\ell$ exists for all $\ell\in\R$. 
If $\wh{U}(0)$ and $\rho'(0)$ both lie in~$(0,\infty)$, then $\theta_X^\ell\in (0,\infty)$. 

\item
\textbf{Absolutely summable correlations.}
Suppose that 
$\int_{\R}|r(t)|\, dt<\infty$
and $\int_{\R} r(t)dt>0$. Then 
$\rho\in \mathcal{M}$ with $\rho'(0)\in (0,\infty)$, and Theorem~\ref{thm: main exist} implies $\theta_\rho^\ell \in (0,\infty)$ for every $\ell$.
Prior to our work, such results were only known under the additional assumption that $r\ge 0$, see \cite{AM, DM}.

\end{enumerate}

Additional examples could be generated on noting that the class of measures $\cL\cap \cM$ is closed under addition, truncation, and convolution.

\subsection{Outline of the paper}

The paper is organized as follows.
In Section~\ref{sec: prelim} we present various tools needed in our proofs.
In Section~\ref{sec: key lemmas} we prove the comparison results: Lemmata \ref{lem: level cont}, \ref{lem: convolution} and \ref{lem: meas cont}.
The rest of the paper is dedicated to the proofs of Theorems 1--5.
The order in which these were presented is not the order of their establishment.
In Section~\ref{sec: singular} we prove Theorem~\ref{thm: nontrivial ball} concerning ball exponents and singular measures.
In Section~\ref{sec: main thm} we prove the existence of a persistence exponent as stated in Theorem~\ref{thm: main exist}, and
the continuity of ball and persistence exponents as stated in Theorem~\ref{thm: cont}; we also provide a class of non-existence examples.
In Section~\ref{sec: s thm} we prove monotonicity and indifference of the exponents to the singular part, namely, Theorem~\ref{thm: singular}.
Lastly, in Section~\ref{sec: sample} we prove Theorem~\ref{thm: sample} regarding convergence of persistence exponents under sampling; an example of non-convergence is also provided.

\section{Preliminaries}\label{sec: prelim}
In this section we collect tools and observations that will serve us in the rest of the paper. 
Throughout the paper, we denote by $f_\rho$ the GSP corresponding to a spectral measure $\rho$.
The notation $\rho_L$ is used for the restriction of a measure $\rho$ onto the interval $[-L,L]$.
We use both $\cF[\rho]$ and $\wh{\rho}$ to denote the Fourier transform of a finite measure $\rho$.

\subsection{Gaussian measures on Euclidean spaces}
In this section we recall several classical properties of Gaussian measures on $\R^d$.
We start with standard estimates of the one-dimensional Gaussian distribution (see, e.g., \cite[Lem. 3.13]{FFN}).
\begin{lem}\label{lem: tail}
Let $Z\sim \cN (0,1)$.
 For all $x>0$ we have:\\[-10pt]
\begin{flalign*}
&\mathrm{(a)}&\frac 1{\sqrt{2\pi}}\left(\frac 1 x - \frac 1{x^3}\right) e^{-x^2/2} &\le  \P(Z>x) \le \frac 1{\sqrt{2\pi}}\frac 1 x  e^{-x^2/2},&& \\
&\text{\phantom{(a)}In particular, for $x\ge 2$:}\hspace{-110pt}& e^{-x^2} &\le \P(Z> x)\le e^{-x^2/2}.&&\\
&\mathrm{(b)}&\hfil\displaystyle\sqrt{\frac{2}{\pi}} x e^{-x^2/2} &\le \P(|Z| \le  x ) \le x, &&\\
&\text{\phantom{(a)}In particular, for $0<x\le 1$:}\hspace{-110pt}&\frac 1 4 x &\le  \mathbb{P}(|Z|\leq x )\le x.&&
\end{flalign*}

\end{lem}

We continue with a few facts about Gaussian measures and intersections or homothety.

\begin{obs}\label{obs: stretch}
Let $\g_d$ be a centered Gaussian measure in $\R^d$, and let $K\subset \R^d$ be a convex domain containing the origin. Then for any $\al \ge 1$,
$\g_d(\al K) \le \al^d \g_d(K)$.
\end{obs}
 
\begin{proof}
Let $(\Theta,r)$ denote the polar coordinates in $\R^d$ (where $\Theta = (\theta_1,\dots,\theta_{d-1})\in \mathbb{S}^{d-1}$).
We have for any convex body $K$:
\begin{align*}
\g_d(K)  = \int_{\mathbb{S}^{d-1}} \left\{ \int_0^{R_K(\Theta)} r^{d-1}g_{\Theta}(r) \ dr \right\} J(\Theta) d\Theta,
\end{align*}
where
 $g_{\Theta}(r)$ the density of $\g_d$ on the ray defined by $\Theta$,
 $R_K(\Theta)$ is the radius of $K$ in direction $\Theta$, and
$r^{d-1} J(\Theta)$ is the appropriate Jacobian.
Using a simple change of variable we have:
\begin{align*}
\g_d(\al K) & = \int_{\mathbb{S}^{d-1}} \left\{ \int_0^{\al R_{K}(\Theta)} r^{d-1}  g_{\Theta}(r) \ dr \right\} J(\Theta) d\Theta
\overset{ [ r=\al s  ] }{=}
\al^d \int_{\mathbb{S}^{d-1}} \left\{ \int_0^{R_K(\Theta)} s^{d-1} g_{\Theta}(\al s)  ds \right\} J(\Theta) d\Theta
\\ & \le \al^d \int_{\mathbb{S}^{d-1}}\left\{ \int_0^{R_K(\Theta)} s^{d-1}  g_{\Theta}(s)  ds \right\} J(\Theta) d\Theta   = \al^d \g_d(K),
 \end{align*}
where the inequality is due to the fact that $g_\Theta(s)$ is decreasing in $s\in [0,\infty)$.
\end{proof}

We shall employ the Khatri-Sidak's inequality (see \cite{Sidak68} or \cite[Ch. 2.4]{LB11}). This classical result is a particular case of the celebrated Gaussian correlation inequality for general convex sets, proved by Royen~\cite{Roy14} in 2014.

\begin{prop}[Khatri-Sidak's inequality]\label{prop: KS}
If $(Z_1,\dots,Z_d)$ is a centered Gaussian vector in $\R^d$, then for any $\{ \ell_j\}_{j=1}^d\subset (0,\infty)$ one has
$
\P\left( \bigcap_{j=1}^{d} \big\{ |Z_j|\le \ell_j \big\} \right) \ge \prod_{j=1}^d \P\big( |Z_j| \le \ell_j\big).
$
\end{prop}

An immediate consequence is the following. 
\begin{cor}\label{cor: GCI cor}
Let $f$ be a Gaussian process which is almost surely continuous on an interval $[0,a+b]$. Then for every $\ell>0$,
\[
\P\left( \sup_{[0,a+b]} |f | <\ell\right)  \ge  \P\left( \sup_{[0,a]} |f | <\ell\right)  \P\left( \sup_{[a,a+b]} |f | <\ell\right) .
\]
\end{cor}
In the case of stationarity, Corollary~\ref{cor: GCI cor} together with Fekete's subadditivity lemma yield the existence of the ball exponent. As per \eqref{eq: per}, denote
\begin{equation}\label{eq: ball}
\psi_\rho^\ell(T) := -\frac 1 {T}\log \P\left( \sup_{t\in [0,T]}|f_\rho(t)|<\ell\right).
\end{equation}

\begin{cor}\label{cor: ball exist}
For any spectral measure $\rho\in \cL$ and any $\ell>0$, the ball exponent $\psi_\rho^\ell:=\lim_{T\to\infty}\psi_\rho^\ell(T)$ exists and lies in $[0,\infty)$.
\end{cor}

Proposition~\ref{prop: KS} is also the main ingredient in the proof of the following result.
\begin{prop}\label{prop: slabs}
Let $\g_d$ be the standard Gaussian measure in $\R^d$.
Then, setting $\kappa=0.05$, for any $n\ge d$ and any collection of unit vectors $u_1,\dots,u_n\in \mathbb{S}^{d-1}$, we have
\[
\g_d\left( \bigcap_{j=1}^n  \{ x\in\R^d: |\langle x,u_j \rangle|\le 1 \}  \right)   \ge \left(  \frac{\kappa}{ \sqrt{1+2\log \frac {n}{d}}}\right)^d.
\]
\end{prop}

The analogue of Proposition~\ref{prop: slabs} in which $\gamma$ is replaced with the Lebesgue measure is a theorem by Ball and Pajor~\cite{BP}. The Gaussian case should follow from the Euclidean one (allowing for different constants), however, we give here a direct Gaussian proof suggested to us by Ori Gurel-Gurevitch.

\begin{proof}
Denote $S_j =  \{ x\in\R^d: |\langle x,u_j \rangle|\le 1 \}$. Let $\al>1$ (to be chosen later). By Observation \ref{obs: stretch} and Proposition \ref{prop: KS} we have:
\begin{align*}
\g_d\Big(\bigcap_{j=1}^n S_j \Big) = \g_d \Big(\frac {1}{\al} \bigcap_{j=1}^n (\al S_j ) \Big)
\ge\al^{-d} \ \g_d \Big( \bigcap_{j=1}^n (\al S_j) \Big) \ge \al^{-d} \g_d \Big( \al S_1 \Big)^n = \al^{-d} \g_1([-\al,\al])^n,
\end{align*}
where $\g_1$ is the standard Gaussian measure in $\R$. Set $\al = \sqrt{1 + 2\log \left(\frac n d\right)}$.
The result will follow once we show that
$
\g_1([-\al,\al])^{n/d} \ge \kappa
$
for all $n\ge d$ and some constant $\kappa>0$ which is independent of $n$ and $d$.
If $\frac n d < e^2$, we bound by
\[
\g_1([-\al,\al]) ^{n/d} \ge
\g_1([-1,1])^{e^2},
\]
while if $\frac n d\ge e^2$ we apply the first part of Lemma \ref{lem: tail} to get
\[
\g_1([-\al,\al]) ^{n/d} \ge \left( 1-2\ \P \left(\cN(0,1) \ge \al \right) \right)^{n/d}
\ge \left(1-2 e^{-\al^2/2}\right)^{n/d}
\ge \left( 1 - 2\frac d n \right)^{n/d} \ge e^{-2}.
\]
Taking $\lambda =0.05< \min\left\{\g_1([-1,1])^{e^2}, e^{-2}\right\}$, the result follows.
\end{proof}

\subsection{Operations on GSPs}

In this section we compute
the spectral measure of a GSP which was generated from another GSP by
a simple operation.

\begin{obs}\label{obs: emp conv}
Assume that $h\in L^2_\rho(\R)$. Then $f_\rho*\wh{h}$ is a GSP with covariance kernel \\ $\widehat{\rho}*\widehat{h}*(\widehat{h}(-\cdot))$ and spectral measure
$|h(\lm)|^2 d\rho(\lm)$.
\end{obs}

\begin{proof}
Denote $H = \wh{h}$.
The random process $W(x) = (f_\rho*H)(x) = \int f_{\rho}(t) H(x-t) dt$ is Gaussian, with covariance kernel given by:
\begin{align*}
\E[W(x)W(y)] & =
\E\left[ \int f_\rho(x-t) H(t) \ dt \int f_\rho(y-s) H(s) \ ds  \right]
\\ & =
\iint \E[ f_\rho(x-t) f_\rho(y-s)] H(t)H(s) \ dt \ ds
\\ & =
\iint \wh{\rho}(x-y+s-t) H(t)H(s) \ dt \ ds
\\ & = \int (\wh{\rho}*H) (x-y+s) H(s) \ ds
\\ &  = \left(\wh{\rho} * H*(H(-\cdot) \right) (x-y) = \FF\left[ |h|^2 \rho\right] (x-y).
\qedhere
\end{align*}
\end{proof}

 {
\begin{obs}\label{obs: sample}
For any $\ep>0$,
the sequence $\{f_\rho(t)\}_{t\in\ep \Z}$
 has the spectral measure $\rho^*_\ep$, 
 supported on $[-\frac{\pi}{\ep},\frac{\pi}{\ep}]$ and  given by:\footnote{In this sense, $\rho_\ep^*$ is achieved as a result of \emph{folding} the spectral measure $\rho$ to the interval $[-\tfrac {\pi}{\ep}, \tfrac{\pi}{\ep}]$.}
  $$\forall \text{  }\mathrm{ interval}\text{  } I \subseteq[-\tfrac{\pi}{\ep},\tfrac{\pi}{\ep}]: \quad \rho^*_\ep(I) = \rho\left( \bigcup_{n\in \Z } ( I + \tfrac{ 2\pi}{\ep} n ) \right).$$
\end{obs}
}
\begin{proof}
  $\rho^{*}_\ep$ is the unique spectral measure supported on $[-\tfrac {\pi}{\ep},\tfrac {\pi}{\ep}]$ such that $\cF[\rho^{*}_\ep](j)=\cF[\rho](j)$ for any $j\in\ep\Z$.
\end{proof}

\subsection{Decompositions of GSPs}
The following claims provide two types of decompositions of a GSP into independent components.
The first is a series representation, which may be found in \cite[Claim 3.8]{FFN}.

\begin{lem}[Hilbert decomposition]\label{lem: hilbert}
Let $\rho \in \cL$
and let $\{\p_n\}$ be an orthonormal basis in $\mathcal{L}^2_\rho$ which satisfies, for every $n\in\N$, $\p_n(-\lambda)=\overline{\p_n(\lambda)}$.
Denote $\Phi_n(t) = \int_\R e^{-i\lm t}\p_n(\lm) d\rho(\lm)$. Then
\[
f_\rho \overset d{=}
 \sum_n \zeta_n 
 \Phi_n, \quad \zeta_n \sim \calN(0,1) \text{ 
 }\mathrm{i.i.d.}
\]
\end{lem}

Lemma~\ref{lem: hilbert} has the following useful consequence.

\begin{cor}[one component]\label{cor: one decomp}
Let $\rho\in\cL$ and let $\p \in \mathcal{L}^2_\rho$ be a real symmetric function such that $\|\p\|_{\mathcal{L}^2_\rho}=1$. Write $\Phi (t) = \int_\R e^{-i\lm t}\p(\lm) d\rho(\lm)$, then we have the decomposition
\[
f_\rho \overset d{=} \zeta  \Phi\oplus g,
 \]
 where $\zeta\sim \calN (0,1)$ and $g$ is a Gaussian process which is independent of $\zeta$.
\end{cor}

The second is the spectral decomposition which appeared in \cite[Obs. 1]{FF}.

\begin{lem}[spectral decomposition]\label{lem: spec decomp}
If $\rho_j$ is a spectral measure for $j\in \{0,1,2\}$ and $\rho_0 = \rho_1 + \rho_2$, then
$f_{\rho_0} \overset{d}{=} f_{\rho_1} \oplus f_{\rho_2}$.
\end{lem}
 
One application of the spectral decomposition is the following simple yet useful lemma.
Recall the notation~\eqref{eq: per} and~\eqref{eq: ball},

\begin{lem} \label{lem: ball-level}
For any spectral measure $\rho, \nu$ and any $\ell\in \R$, $\delta>0$ and $T>0$, we have the following.
\begin{enumerate}[(a)]

\item
$\psi^{\ell}_{\rho+\nu}(T) \le \psi^{\ell-\delta}_{\rho}(T)+\psi^\delta_{\nu}(T) $.

\item
$\theta^{\ell}_{\rho+\nu}(T) \le \theta^{\ell+\delta}_{\rho}(T)+\psi^\delta_{\nu}(T) $.

\end{enumerate}
\end{lem}
\begin{proof}

\textbf{Part (a).}
Using Lemma~\ref{lem: spec decomp}, we have 
\[\P\left(\sup_{t\in [0,T]}|f_{\rho+\nu}(t)|<\ell \right)=\P \left( \sup_{t\in [0,T]}\left|f_{\rho}(t)\oplus f_{\nu}(t)\right|<\ell \right)\ge
\P\left(\sup_{t\in [0,T]}|f_\rho(t)|<\ell-\delta\right) \P\left(\sup_{t\in [0,T]}|f_\nu(t)|<\delta\right),\]
which upon taking $\log$ and dividing by $T$ yields the desired inequality.

\textbf{Part (b).}
Again using Lemma~\ref{lem: spec decomp}, we have
\[\P\left(\inf_{t\in [0,T]}f_{\rho+\nu}(t)>\ell \right)=\P \left( \inf_{t\in [0,T]}\left( f_{\rho}(t)\oplus f_{\nu}(t)\right)>\ell \right)\ge
\P\left(\inf_{t\in [0,T]}f_\rho(t)>\ell+\delta\right) \P\left(\sup_{t\in [0,T]}|f_\nu(t)|<\delta\right),\]
which upon taking $\log$ and dividing by $T$ yields the desired inequality.

\end{proof}

\subsection{Classical Gaussian tools}\label{sec: famous}
In this section we recall classical tools from the theory of Gaussian processes.
We start with the celebrated Slepian's lemma, see \cite[Thm. 2.1.2]{AT} or \cite{slepian62}.

\begin{prop}[Slepian]\label{prop: slep}
Let $X$ and $Y$ be centered Gaussian processes on $I\subset \R $. Suppose that
\[
\E[ X_t X_s ] \le \E[Y_t Y_s], \quad \E[X_t^2] = \E[Y_t^2],  \quad \forall t,s\in I.
\]
Then for any $\ell\in \R$ one has
\[
\P\left(\sup_I X > \ell \right) \le \P\left(\sup_I Y > \ell \right).
\]
\end{prop}

The following famous concentration bound is due to Borell and Tsirelson-Ibragimov-Sudakov, see \cite[Thm. 2.1.1]{AT}.
\begin{prop}[Borell-TIS]\label{prop: BTIS}
Let $X$ be a centered Gaussian process on $I$ which is almost surely bounded. Then for all $u>0$ we have:
\[
\P\left(\sup_I X  - \E \sup_I X >u \right) \le \exp\left(-\frac{u^2}{2s^2}\right),
\]
where $s^2= \sup_{t\in I} \var X(t)$.
\end{prop}

The expected supremum of a Gaussian process is often bounded using Dudley's metric-entropy method \cite[Thm. 1.3.3]{AT}.
For a Gaussian process $H$ on an interval $I$, we let
\begin{equation}\label{eq: metric}
d_H(a,b) :=\sqrt{ \E ( H(a) - H(b) )^2 }, \quad a,b\in I
\end{equation}
be the canonical semi-metric induced by $H$, and denote $\mathrm{diam}_H(I) = \sup_{a,b\in I} d_H(a,b)$.
For any $x>0$, the covering number $N_H(x)$ is the minimal number of $d_H$-balls of radius $x$ which cover $I$.

 \begin{prop}[Dudley's bound]\label{prop: Dudley}
There exists a universal constant $K>0$ such that for any Gaussian process $H$ on $I$ we have
\begin{equation*}
\E \sup_I H \le K \int_0^{\mathrm{diam}_H(I)} \sqrt{\log N_H(x)}\,  dx.
\end{equation*}
\end{prop}

Lastly we recall a comparison between ball probabilities due to Anderson \cite[Ch. 2.3]{LB11}.

\begin{prop}[Anderson]\label{prop: and}
Let $X, Y$ be two independent, centered Gaussian processes on $I$. Then for any $\ell>0$,
\[
\P\Big(\sup_{I} |X \oplus Y| \le \ell\Big) \le \P\Big(\sup_{I} |X|\le \ell\Big).
\]
\end{prop}

\subsection{Supremum}

In this section we apply tools from Section~\ref{sec: famous} in order to estimate events concerning the supremum a GSP whose spectral measure is in the class $\cL_{\beta,B}$.
\begin{lem}\label{lem: log bound on r}
Let $\rho\in \cL_{\beta,B}$. Then:
$$\forall t\in (0,1): \quad 0\le r(0)-r(t) \le  \frac{3B}{\log^{1+\beta} (1/t)}.$$
\end{lem}

\begin{proof}
For any $t\in (0,1)$ we have:
\begin{align*}
r(0)-r(t)&=\int_\R\Big(1-\cos(\lm t) \Big) d\rho(\lm)
= \left(\int_{|\lm|< \frac 1 {\sqrt t}} +\int_{|\lm|\ge \frac 1 {\sqrt t}}\right)  \Big(1-\cos(\lm t) \Big) d\rho(\lm)
\\&\le \int_{|\lm|< \frac 1 {\sqrt t}} \frac{\lambda^2t^2d\rho
(\lm)}{2}+
2\int_{|\lm|\ge \frac 1 {\sqrt t}} \ d \rho(\lambda)
\\ &\le  {
\frac{t}{2}\int_{|\lm|< \frac 1 {\sqrt t}} d\rho(\lm)+
\frac{2}{\log^{1+\beta}(1/t)}\int_{|\lm|\ge \frac 1 {\sqrt t}}\log^{1+\beta}|\lambda| \ d \rho(\lambda)}\\
&\le B t + \frac{2B}{\log^{1+\beta}(1/t)}\le  \frac{3B}{\log^{1+\beta} (1/t)}.\qedhere
\end{align*}
\end{proof}

If $\rho$ has a finite second moment, then Lemma~\ref{lem: log bound on r} may be replaced by the following observation.

\begin{obs}\label{obs: comp bound on r}
Suppose that $\rho$ has a finite second moment, then
\[
\forall t\in \R: \quad 0\le r(0)-r(t)\le  {\frac {t^2} 2\int_\R \lm^2 d\rho(\lm)}.
\]
 In particular, if $\rho$ is supported on $[-D,D]$, then
$0\le r(0)-r(t) \le \tfrac {1}{2} D^2  r(0) t^2.$
\end{obs}

\begin{proof}
$
r(0)- r(t) =\int_{\R} (1-\cos(\lm t) )d\rho(\lm)
 \le \int_{\R} \frac {\lm^2 t^2}{2}d\rho(\lm) $.
\end{proof}

\begin{lem}\label{lem: sup stat}
There exists $C=C(\beta, B)$ such that for any $\rho\in \cL_{\beta,B}$ and any $h\le 1$,
\[
\E \sup_{[0,h]} |f_\rho| < C \sup_{t\in [0,h]} |r(0)-r(t)|^{\frac {\beta}{2(1+\beta)}}.
\]
\end{lem}
\begin{proof}
Let $d_f$ be the canonical semi-metric induced by $f=f_\rho$ (defined via~\eqref{eq: metric}).
Using stationarity and Lemma \ref{lem: log bound on r} we have, for any $s\in\R$ and $t\in (0,1)$,
\[
d_f(s,s+t) =  \sqrt{2 |r(0)-r(t)|} \le \frac{\sqrt{6 B} }{\log^{(1+\beta)/2}(1/t)}=:\psi(t).
\]
This yields that for an interval $I = [0,h]$ we have
\[
N_f(x) \le \max\left(1, \frac{|I|}{ \psi^{-1}(x)} \right)  = \max\left( 1, h \exp \left[ \left(\frac{6B}{x^2} \right)^{\frac 1 {1+\beta}} \right]\right),
\]
and
\[
\mathrm{diam}_f(I) = \sup_{x,y\in I} d_f(x,y)  \le \sup_{t\in [0,h]} \sqrt{2( r(0) - r(t)) }.
\]
By Dudley's bound (Proposition \ref{prop: Dudley}) there exists a universal constant $K>0$ such that
\begin{align*}
\E \sup_{[0,h]} |f| & < K \int_0^{\mathrm{diam}(I)} \sqrt{ \left( \left(\frac{6B}{x^2} \right)^{\frac 1 {1+\beta}}+\log h \right)_+} \ dx
\le  K(6B)^{\tfrac 1 {2(1+\beta)}} \int_0^{\mathrm{diam(I)}} x^{-\frac 1 {1+\beta}}  \ dx
\\ & \le  C(\beta,B) \sup_{t\in[0,h]} |r(0)-r(t)|^{\frac \beta{2(1+\beta)}}.
\qedhere
\end{align*}
\end{proof}

  \begin{lem}\label{lem: small sup}
There exist $C_1, C_2>0$ (depending only on $\beta, B$) such that for any $\nu\in \cL_{\beta,B}$ and any $m>0$,
\[
\P\left(\sup_{[0,1]} |f_\nu|> \left(C_1 m +C_2 \right) \nu(\R)^{ \frac{\beta}{2(1+\beta)} }   \right)\le 2e^{-m^2/2}.
\]
  \end{lem}
 \begin{proof}
Denote $I=[0,1]$. Note that
\[
\sqrt{\sup_{t\in I} \var f_\nu(t)} =\sqrt{r(0)}=\sqrt{\nu(\R)}\le C_1 \nu(\R)^{\frac{\beta}{2(1+\beta)}},
\]
and by Lemma \ref{lem: sup stat} we have
\[
\E \sup_{I} |f_\nu| < C_2 \nu(\R)^{\frac {\beta}{2(1+\beta)}},
\]
where $C_1$ and $C_2$ depend only on $\beta, B$.
By Proposition \ref{prop: BTIS} (Borell-TIS theorem), we have for any $m>0$:
\[
\P\left(\sup_{I} f_\nu > ( C_1 m + C_2) \nu(\R)^{\frac{\beta}{2(1+\beta)}} \right) \le \P\left(\sup_{I} f_\nu > \E\sup_{I} f_\nu + m\sqrt{\sup_{t\in I} \var f_\nu(t)} \right)\le 
 e^{-m^2/2}.
\]
The statement then follows using symmetry of $f_\nu$ and $-f_\nu$ and a union bound.
  \end{proof}

\begin{lem}\label{lem: win ball} 
Let $W$ be a centered Gaussian process on $[0,T]$, satisfying
\begin{align}
&\sup_{t\in[0,T]} \var W(t) \le \frac{\ep}{T}, \label{eq: var W} \\
& \sup_{t\in[0,T]} \var\left(W(t+h)-W(t) \right) \le a^2 h^2, \label{eq: dif W}
\end{align}
for some $\ep \in (0,1)$ and $a>0$ and all $h>0$.
 {Then there exists a universal constant $c>0$ such that for all $\delta>0$ and $T>c(a^{-1}+\delta^{-1}+1)$} we have:
\[
\P\left(\sup_{[0,T]} |W | > \delta\right)
 \le
2 e^{-\frac {\delta^2}{8\ep} T}.
\]
\end{lem}

\begin{proof}
Let $d_W$ be the canonical semi-metric induced by $W$, as in \eqref{eq: metric}.
Using \eqref{eq: dif W} we have
\[
d_W(t,t+h) =  \sqrt {\var(W(t+h)-W(t)) }\le a h,
\]
for any $t,h>0$.

This yields that the covering number $N_W(x)$ of the interval $[0,T]$ in the $d_W$-metric obeys $N_W(x) \le \max(\frac{ a T}{x},1)$ for all $x>0$.
Moreover, using \eqref{eq: var W} we have 
\[
\mathrm{diam}_W([0,T]) = \max\limits_{x,y\in [0,T]} d_W(x,y)  \le \sqrt{2 \frac{\ep}{T} }\le \sqrt{\frac 2 T}.\]
Consequently, by Dudley's bound (Proposition \ref{prop: Dudley}), for $T>(2/a^2)^{1/3}$:

\begin{align*}
\E \sup_{[0,T]} |W|<K \int_0^{\sqrt{\frac{2}{T}}} \sqrt{\log \frac{aT }{x}}dx
\end{align*}
where $K>0$ is a  {universal constant. Hence there exists $c>0$ such that if $T>c(a^{-1}+\delta^{-1}+1)$} we obtain
\begin{equation}\label{eq: new Esup->0}
\E \sup_{[0,T]} |W| \le \frac {\delta}{2}.
\end{equation}
Using Borell-TIS (Proposition \ref{prop: BTIS}) with \eqref{eq: var W} and \eqref{eq: new Esup->0} we obtain that, for  {$T>c(a^{-1}+\delta^{-1}+1)$},
\[
 \P\left(\sup_{[0,T]} |W| > \delta\right) \le
 2 \P\left( \sup_{[0,T]}W >\delta\right)\le
 2 \P\left(\sup_{[0,T]} W - \E\sup_{[0,T]} W > \frac{\delta}{2}\right)
 \le
2 e^{-\frac {\delta^2}{8\ep} T}.\qedhere
\]
\end{proof}

\subsection{Bounds on ball and persistence exponents}
In this section we present a priori bounds on ball and persistence probabilities, which hold uniformly for spectral measures in the class $\cL_{\beta,B}$ or $\cM_{\al,A}\cap \cL_{\beta,B}$ and a given level $\ell$.
The first such bound is a slightly stronger version of
\cite[Lemma 3.12]{FFN} or~\cite{Tal}, 
as we assume a finite log-moment instead of a finite polynomial moment.

\begin{lem}\label{lem: lower ball}
There exists $C=C(\beta,B,\ell)\in (0,\infty)$
such that for all $\rho\in \cL_{\beta,B}$ and $T\ge 1$:
\[
\P\left(\sup_{[0,T]} |f_\rho| < \ell \right)  \ge e^{-C T}, \textup{ or equivalently  }
\psi_\rho^\ell(T) \le C.
\]
\end{lem}
\begin{proof}
By Khatri-Sidak's inequality (Cor.~\ref{cor: GCI cor}) we have, for any $h\in (0,1)$,
\begin{equation}\label{eq: T/h}
\P\left(\sup_{[0,T]} |f_\rho| < \ell \right) \ge \P\left(\sup_{[0,h]} |f_\rho| < \ell \right)^{\lceil T/h \rceil}.
\end{equation}
Combining Lemma~\ref{lem: log bound on r} and Lemma~\ref{lem: sup stat} we have
\[
\E \sup_{[0,h]} |f_\rho| \le C \log^{-\beta/2}\left(\tfrac 1 h\right).
\]
Consequently there exists $h\in (0,1)$, depending on $\ell, \beta, B$, such that $\E \sup_{[0,h]} |f_\rho|  < \frac{\ell}{2}$.
An application of Markov's inequality gives
$
\P\left( \sup_{[0,h]} |f_\rho| > \ell \right) \le \frac 1 2.
$
For this $h$, inequality \eqref{eq: T/h} gives
\[
\P\left(\sup_{[0,T]} |f_\rho| < \ell \right)\ge 2^{-\lceil T/h \rceil}. \qedhere
\]
\end{proof}

The next result is a somewhat more general version of
\cite[Theorem 2]{FF}.

\begin{lem}\label{lem: FF LB}
There exists $C=C(\al,\beta,B,\ell)\in (0,\infty)$ such that for all $\rho \in \cM_{\al, A} \cap \cL_{\beta,B}$ and all $T\ge \max(\tfrac 1 A,1)$:
\[
\per{\rho}  \ge e^{-C T}, \quad \textup{or equivalently}\quad \ttt{\rho} \le C.
\]
\end{lem}

\begin{proof}
Recall the notation $\rho_L= \rho|_{[-L,L]}$ for the restriction of a measure $\rho$ to the interval $[-L,L]$.          
For a fixed (arbitrary) $m>0$ we have
\begin{equation*}
\ttt{\rho} \le \tttl{ \rho_{1/T} }{\ell+m} +\psi_{\rho-\rho_{1/T}}^m (T)
\le  \tttl{ \rho_{1/T} }{\ell+m} + \psi_\rho^m(1),
\end{equation*}
where the first inequality holds by Lemma~\ref{lem: ball-level}(b), and the second one follows from the inequalities by Anderson (Proposition~\ref{prop: and}) and Khatri-Sidak (Proposition~\ref{prop: KS}).
The covariance function corresponding to $ \rho_{1/T} $ is
\[
\FF[\rho_{1/T}](t) = \int_{-1/T}^{1/T} \cos(\lm t) d\rho(\lm) \ge \rho([-\tfrac 1 T, \tfrac 1 T]) \cos(\tfrac 1 T t),
\]
and equality holds at $t=0$.
Notice that the RHS of the last inequality is the covariance kernel of the process
\[
a_T \cos(\tfrac 1 T t) + b_T\sin (\tfrac 1 T t), \quad \textup{ where  } a_T,b_T \sim \cN\left(0,\rho\left([-\tfrac 1 T, \tfrac 1 T]\right) \right) \text{  are i.i.d.}
\]
By Slepian's inequality (Proposition~\ref{prop: slep}) we have
\begin{align*}
\tttl{ \rho_{1/T} }{\ell+m}
&\le -\frac 1 T \log \P\left( \forall t\in [0,T]:  \: a_T \cos(\tfrac 1 T t) + b_T\sin (\tfrac 1 T t)  > \ell+m\right) \notag
\\ &
\le -\frac 1 T \log \P(a_T > \tfrac{\ell+m}{\cos 1}) - \frac{1}{T}\log \P\left( b_T>0\right)
\notag
\\ & \le -\frac 1 T \log \P\left(\sqrt{\rho([-\tfrac 1 T,\tfrac 1 T])} \, Z> \tfrac{\ell+m}{\cos 1}\right) + \frac{\log 2 }{T},
\end{align*}
where $Z\sim \cN(0,1)$.
We have thus proved that, for $T\ge \log 2$,
\begin{equation}\label{eq: aux}
\theta_\rho^\ell(T) \le \psi_\rho^m(1)+1 - \frac 1 T \log \P\left(\sqrt{\rho([-\tfrac 1 T,\tfrac 1 T])} \, Z> \tfrac{\ell+m}{\cos 1}\right).
\end{equation}
Since $\rho\in \cL_{\beta,B}$, by Lemma~\ref{lem: lower ball}, it holds that $\psi_\rho^m(1) \le C(\beta,B) <\infty$. To see that the last term is bounded, recall that $\rho([-\tfrac 1 T,\tfrac 1 T]) \ge \frac {2\al}{T}$ whenever $T\ge \frac 1 A$, and thus
\[
- \frac 1 T \log \P\left(\sqrt{\rho([-\tfrac 1 T,\tfrac 1 T])} \, Z> \tfrac{\ell+m}{\cos 1}\right)\le -\frac 1 T \log \P\left( Z  > \frac{\ell+m}{\sqrt{2\al}\cos 1} \sqrt T\right) \le
 \frac 1{2\al}\left(\frac{\ell+m}{\cos 1}\right)^2.
\]
The last step uses Lemma~\ref{lem: tail} and assumes that $m$ is such that $\tfrac{\ell+m}{\sqrt{2\al}\cos 1}>2$ along with $T\ge 1$.
\end{proof}
The reverse direction of Lemma~\ref{lem: FF LB} requires some extra assumptions.

\begin{lem}\label{lem: FF UB}
Let $\rho \in \cL_{\beta,B}\cap \cM_{(0,\al'),\Alt}$ and $\ell\in\R$.
 Suppose further that $d\rho \ge  m\textbf{1}_{E}(\lm) d\lm$ where $d\lm$ is the Lebesgue measure, $m>0$ and $E$ is a Lebesgue-measurable set of positive measure.
Then there exist $C=C(\al', \Alt, \beta, B,\ell, m, |E|)>0$ such that for all $T\ge 1$:
\[
\per{\rho}  \le e^{-C T}, \quad \textup{or equivalently}\quad \ttt{\rho}\ge C.
\]
If $\ell\ge 0$, then the constant $C$ satisfies $\lim_{\al'\to 0} C(\al')=\infty$.
\end{lem}

\begin{proof}
This inequality was proved in \cite[Prop. 3]{FFN}, which is a corollary of Theorem 5.1 there. The assumption throughout that paper is that $\int |\lm|^\delta d\rho(\lm)<\infty$ for some $\delta>0$ and that $\ell=0$; however, the proof in our case (corresponding to $\gamma=1$ and $b=\al'$ there) applies as soon as $\rho\in \cL$ and for any level $\ell\in\R$.
The dependence of $C$ on the parameter $\al'$ follows from \cite[Remark 2]{FFN}.
\end{proof}

 \begin{prop}\label{prop: low explode}
 Let $\rho\in \cL$ and $\ell\in\R$.
If $\lim\limits_{\ep\to 0} \frac{\rho(-\ep,\ep)}{2\ep} = \infty$ then
 $\theta^\ell_\rho = 0$.
 \end{prop}

\begin{proof}
Fix a large parameter $m>0$.
Using Lemma~\ref{lem: tail} and the fact that $\lim_{T\to\infty} T\rho([-\frac 1 T,\frac 1 T])=\infty$, we deduce that
\[
\limsup_{T\to\infty}\frac 1 T \log \P\left(\sqrt{\rho([-\tfrac 1 T,\tfrac 1 T])} \, Z> \tfrac{\ell+m}{\cos 1}\right) =0.
\]
Plugging this into \eqref{eq: aux} we obtain
\[
0\le \limsup_{T\to\infty}\theta_\rho^\ell(T) \le \psi_\rho^m(1).
\]
The desired conclusion follows on letting $m\to\infty$.
\end{proof}

\section{Proofs of the comparison lemmata}\label{sec: key lemmas}

\subsection{Proof of Lemma \ref{lem: level cont}: continuity in levels}

Fix $T>0$, $\delta>0$. The inequality $\perl{\mu}{\ell + \delta} \le \perl{\mu}{\ell-\delta}$, or equivalently
$\theta_\mu^{\ell-\delta}(T) \le \theta_\mu^{\ell+\delta}(T)$,
 follows from inclusion of events. It remains to bound the difference
 $\theta_\mu^{\ell+\delta}(T) - \theta_\mu^{\ell-\delta}(T)$.
Let
\[
\p_{T} := \frac{1}{s_{T}}1\!\!1_{[-\frac{1}{T},\frac{1}{T}]},
\quad \text{where    }s^2_{T} := \int_{-1/T}^{1/T}  d\mu(\lm)\ge \frac {2\al}{T},
\]
so that $\p_{T}\in \mathcal{L}^2_{\mu}$ and $\|\p_{T} \|_{\mathcal{L}^2_{\mu}}=1$.
Denote $\psi_{T}(t) =\int_\R e^{-i\lm t}\p_{T}(\lm) d\mu(\lm)$ and
invoke Corollary~\ref{cor: one decomp} to obtain the decomposition
\[
f_{\mu}(t)\overset{d}{=}\zeta \psi_{T} (t)\oplus R_{T}(t),
\]
where $\zeta\sim N(0,1)$ and $R_{T}$ is a centered Gaussian process on $[0,T]$. Note that on the interval $[0,T]$ the function $\psi_{T}$ satisfies
\[
\psi_{T}(t)=\frac{1}{s_T}\int_{-\frac{1}{T}}^{\frac{1}{T}}\cos (t\lm) d\mu(\lm)
\geq \frac{1}{s_T}\cos\left(\tfrac{t}{T}\right) s^2_T\geq 
\cos(1) s_T \ge \frac{1}{2} \cdot \sqrt{\frac{2\alpha}{T}} = \frac{2 }{ c \sqrt{T}},
\]
for $T\ge \frac 1 A$ and $c= \sqrt{\frac{8}{\al}}$.

Denote $\widetilde{f}(t)=(\zeta - c \delta \sqrt{T}) \psi_T(t) \oplus R_T(t)$
and observe that
\begin{equation*}
f_\mu(t)  =\widetilde{f}(t) + c \delta \sqrt{T} \psi_T(t),
\end{equation*}
and $\wt{f}(x) \overset{d}{=} \wt{\zeta}\psi_T(x) \oplus R_T(x)$ with $\wt{\zeta}\sim \mathcal{N}(-c\delta\sqrt{T},1)$.
Since $ c\delta\sqrt{T} \psi_T(x) \ge 2\delta$, we have
\begin{equation}\label{eq: start}
\perl{\mu}{\ell+\delta} = \P( f> \ell+\delta \text{  on } [0,T])
 \ge \P( \widetilde{f} > \ell -\delta \text{  on } [0,T] ).
\end{equation}

By Lemma~\ref{lem: FF LB}, there exists $M\in (0,\infty)$ such that
$ \per{\mu} \ge 4e^{-MT/2}$ for $T\ge 1$. Replacing~$M$ with $\max(M,1)$ if necessary, we may assume that $M\ge 1$. Invoking Part (a) of Lemma \ref{lem: tail},
we obtain $ 2e^{-MT/2} \ge \P(|\zeta| \ge \sqrt{MT})$ for $T\ge 4$, which yields
\begin{equation}\label{eq: level to M}
 \perl{\mu}{\ell-\delta} \ge \per{\mu} \ge 4 e^{-MT/2} \ge 2\P(|\zeta| \ge \sqrt{MT}).
\end{equation}
Denote by $\phi$ the density of the random variable $\zeta$, and by $\wt{\phi}$ the density of $\wt{\zeta}$.
Starting from \eqref{eq: start} and using Radon-Nikodym derivative estimate we obtain
\begin{align*}
\perl{\mu}{\ell+\delta} & \ge \P\Big( \widetilde{f}\ge \ell- \delta \text{  on } [0,T], |{\zeta}|< \sqrt{MT}\Big)
\\ & \ge \inf_{|x|\le \sqrt{MT} } \left| \frac{\phi(x)}{\wt{\phi}(x)} \right| \P\Big( {f}_\mu\ge \ell- \delta \text{  on } [0,T], |\zeta|< \sqrt{MT}\Big)
\\ & \ge  \inf_{|x|\le \sqrt{MT} } e^{  c\delta \sqrt{T} \cdot x}  \left(\P\Big( {f}_\mu\ge \ell- \delta \text{  on } [0,T]\Big)-\P(|\zeta| \ge \sqrt{MT})\right)
\\& \ge  \tfrac 1 2  e^{-c\delta\sqrt{M} T}  \P\Big( {f}_\mu\ge \ell- \delta \text{  on } [0,T]\Big)= \tfrac 1 2 e^{-c\delta\sqrt{M} T}\perl{\mu}{\ell-\delta},
\end{align*}
where the last inequality is due to \eqref{eq: level to M}.
Taking logarithm, this yields
\[
\log \perl{\mu}{\ell-\delta} - \log \perl{\mu}{\ell+\delta}  
\le  {\log 2 + c\delta \sqrt{M}T}\le 
\log 2 + C \delta T,
\]
 {for all 
$T \ge \max\{4,\tfrac 1 A\}$, where
$C=c\sqrt{M}$}.


\subsection{Proof of Lemma \ref{lem: convolution}: smoothing increases persistence}
Denoting $H = \wh{h}$, the assumptions of the lemma are $H\ge 0$, $\int_\R H =1$ and $\sprt H \subseteq [-\tfrac a 2,\tfrac a 2]$.
By Observation \ref{obs: emp conv}, the measure $h^2 \nu$ is the spectral measure of the GSP $f_\nu * H$.
Therefore our objective is to show that

\begin{equation}\label{eq: goal1}
\P\Big(f_\mu \oplus  f_\nu > \ell \text{   on } [0,T+a]\Big) \le  \P\Big(f_\mu \oplus (f_{\nu} *H)>\ell \text{  on } [0,T]\Big) .
\end{equation}

Gaussian measures are well-known to be log-concave (see \cite[Example 2.3]{BM}). In particular, for any Gaussian process $X$ on $[0,T]$,
\begin{align*}
\int H(s) \log \P\Big(X(t) + v(s,t)  > \ell, \: \forall t\in[0,T]  \Big)\, ds \le \log \P\Big(X(t) + \int H(s) v(s,t)  ds > \ell,  \: \forall t\in[0,T]\Big),
\end{align*}
where $v(s,t): \R\times [0,T]\to\R$ is a continuous function in $t$ for every fixed $s\in \R$.
Thus, given a continuous function $v:\R \to \R$, we may apply this with $v(s,t ) = v(t-s)$ to obtain
\begin{equation*}
(\star) := \int H(s) \log \P\Big( X(t) + v(t-s) > \ell, \: \forall t\in [0,T]\Big) \le
 \log \P\Big(X(t) + (H*v)(t) >\ell, \:\forall t\in [0,T]\Big).
\end{equation*}
Assuming $X$ is stationary, and recalling that $\sprt H \subseteq [-\tfrac a 2,\tfrac a 2]$, we have
\begin{align*}
(\star) &= \int_\R H(s) \log \P\Big(X(t+s) + v(t) > \ell, \: \forall t\in [-s, T-s]\Big) \ ds
\\ & \ge \int_\R H(s) \log \P\Big( X(t+s) +v(t)>\ell, \: \forall t\in[-\tfrac  a 2, T+\tfrac a 2 ]\Big) \ ds
\\ &  = \int_\R H(s) \log \P\Big( X(t) +v(t)>\ell, \: \forall t\in[-\tfrac  a 2, T+\tfrac a 2 ]\Big) \ ds
\\ & = \log \P\Big( X(t) +v(t)>\ell, \: \forall t\in[-\tfrac  a 2, T+\tfrac a 2 ]\Big),
\end{align*}

where the last line uses $h(0) = \int_\R H(s) ds = 1$. Putting these together, we obtain
\[
 \P\Big( X(t) +v(t)>\ell, \: \forall t\in[-\tfrac  a 2, T+\tfrac a 2 ]\Big)  \le
 \P\Big(X(t) + (H*v)(t) >\ell, \:\forall t\in [0,T]\Big).
\]
Given a real valued stochastic process $Y$, independent of $X$, with almost-surely continuous path, we may apply this to deduce
to deduce that
\begin{equation*}
 \P\Big( X(t) \oplus Y(t)>\ell, \: \forall t\in[-\tfrac  a 2, T+\tfrac a 2 ]\Big)  \le
 \P\Big(X(t) \oplus (H*Y)(t) >\ell, \:\forall t\in [0,T]\Big).
\end{equation*}
When $Y$ is also stationary, we may replace the interval $[-\tfrac a 2, T+\tfrac a 2]$  in the last inequality with $[0,T+a]$.
Applying this to $X = f_\mu$ and $Y= f_\nu$ we obtain \eqref{eq: goal1} as required.


\subsection{Proof of Lemma \ref{lem: meas cont}: continuity in measure}

\subsubsection{Proof of Part \ref{item: cont a}}
Let $\delta>0$ (to be chosen later). By Lemma~\ref{lem: ball-level}(b) we have:
\begin{equation}\label{eq: A1}
\ttt{\mu +\nu} - \tttl{\mu}{\ell+\delta} \le \psi_{\nu}^\delta(T).
 \end{equation}
By Corollary~\ref{cor: GCI cor}, we have for $T>1$
\begin{equation}\label{eq: A2}
\psi_{\nu}^\delta (T) = -\frac 1{T} \log \P\Big( |f_\nu|<\delta \text{  on } [0,T] \Big) \le
-2 \log \P\Big(\sup_{[0,1]} |f_\nu|<\delta   \Big).
 \end{equation}
Next apply Lemma~\ref{lem: small sup} with $m=\sqrt{2\log (1/\ep)}$, using the fact that $\nu(\R)<\ep$, to get
\[
-\log \P\Big(\sup_{[0,1]} |f_\nu|<(C_1 \sqrt{2\log (1/\ep)}+C_2) \ep^{\frac{\beta}{2(1+\beta)}} \Big) \le 2\ep,
\]
where $C_i=C_i(\beta,B)$ for $i\in\{1,2\}$. Thus choosing $\delta(\ep) = (C_1 \sqrt{2\log (1/\ep)}+C_2) \ep^{\frac{\beta}{2(1+\beta)}}$ we have $\lim_{\ep\to 0} \delta(\ep) =0$ and
\begin{equation}\label{eq: A3}
 -\log \P\Big(\sup_{[0,1]} |f_\nu|<\delta(\ep)   \Big) \le 2\ep.
\end{equation}
Combining \eqref{eq: A1}, \eqref{eq: A2} and \eqref{eq: A3} we obtain:
  \[
  \ttt{\mu+\nu} - \tttl{\mu}{\ell+\delta(\ep) } \le 2\ep.
  \]
  By Lemma \ref{lem: level cont},  {for $T\ge \max\{4,\frac 1 A, \frac 1 \ep\}$ we have}  
  \[
  \tttl{\mu}{\ell+\delta(\ep) }- \tttl{\mu}{\ell} \le C \delta(\ep) +  {\ep \log 2},
  \]
  where $C>0$.
The last two inequalities together yield the desired conclusion.

\subsubsection{An auxiliary result}
For proving Lemma~\ref{lem: meas cont}\eqref{item: cont b} we shall need the following proposition. 
We shall use the class $\cM_{(\al,\al'),A}$ defined in \eqref{eq: classes}, allowing for $\al=0$.

\begin{prop}[truncation]\label{prop: truncate}
Let
$\mu\in\cM_{\al,A}\cap\cL_{\beta,B}$ and $\nu\in \mathcal{M}_{(0,\alt),\Alt}\cap \cL_{\beta,B}$ with $0<D\le 1$ and $0<\alt\le B$.
There exist constants $c_1,c_2,c_3>0$, 
depending only on $\al,\beta, B, \ell$, such that for any $0<\eta <c_1 \alt$, 
$T\ge \max\left\{ \frac{c_2}{\eta},\tfrac 1 A\right\}$ and 
$\frac{c_3}{\eta}\le L\le 
 \frac{\alt}{B} D^{2}
T$, we have
\[
\big(1-\tfrac 2 {L^{1/4}}\big)
\tttlt{\mu+\nu_{\frac L T}\!\!\!}{\ell}{T\big(1-\tfrac 2 {L^{1/4}}\big)}
\le \ttt{\mu+\nu} + C_{\eta,\alt},\]
where $\lim_{\eta\to 0 }C_{\eta,\alt}=0$.
\end{prop}

\begin{proof}
Set $M:=L^{1/4}$.
 Denote $h(\lm) = \sinc^2\left(\tfrac T {M} \lm\right)$  so that $H(x) = \wh{h}(x) =\frac{M}{T}(1- \frac M T |x|)_+$.
Since $H$ is compactly supported, we may apply Lemma~\ref{lem: convolution} with $a=\frac{2T}{M}$ to get:
\begin{equation}\label{eq: smoothing with h}
\P\Big(f_{\mu+ h^2 \nu}> \ell \text{  on  } I_T\Big)\ge \P\Big(f_{\mu+\nu} > \ell \text{  on  } [0,T]\Big),
\end{equation}
where $I_T := [0,(1-\tfrac 2 M)T]$.
Recalling the spectral decomposition (Lemma~\ref{lem: spec decomp}) and the notation $\nu_{\frac{L}{T}} = \nu|_{[-\frac{L}{T},\frac{L}{T}]}$, the left-hand-side of~\eqref{eq: smoothing with h} is bounded by
\begin{align}\label{eq: two parts}
\P\Big(f_{\mu+ h^2 \nu}> \ell \text{  on  } I_T\Big)&=
\P\Big(f_{\mu+h^2 (\nu-\nu_{{L}/ T}) + h^2 \nu_{{L}/{T}}}> \ell \text{  on  } I_T\Big) \notag
\\ & \le \P\Big( f_{\mu+h^2 \nu_{\frac L T}} >\ell-\eta\sqrt{\alt} \text{   on   } I_T \Big) + \P\left(\sup_{I_T} |f_{h^2(\nu-\nu_{L/T})}| > \eta\sqrt{\alt} \right),
\end{align}
where $\eta<1$ is an auxiliary error parameter.
By Lemma~\ref{lem: FF LB}, there exists a fixed $R>0$ such that for all $T\ge \max(1,\tfrac 1 A)$,
\begin{equation*}
\P\Big( f_{\mu+h^2 \nu_{\frac L T}} > \ell  \text{  on  } I_T \Big)
\ge e^{-R T}.
\end{equation*}
Denoting $g=\frac 1 {\sqrt{\alt}}f_{h^2(\nu-\nu_{L/T})}$, we have for any $t\in\R$,
\begin{align*}
 \var(g(t))&=
\frac{2}{\alt}\int_{L/T}^\infty h^2 d\nu
=\frac{2}{\alt} \int_{L/T}^\infty  \sinc^4\left(\tfrac{T}{M}\lm\right) d\nu(\lm)  \le 
\frac{2M^4}{\alt \pi^4 T^4} \int_{L/T}^\infty \frac{d\nu(\lm)}{\lm^4} & \gray{|\sinc(x)|\le \frac {1}{\pi |x|}}
\\ & =\frac{2L}{\alt \pi^4 T^4}\left[ \frac{\nu([\tfrac L T,\lm])}{\lm^4}\Big|_{L/T}^\infty + 4  \int_{L/T}^\infty 
\frac{\nu([\tfrac L T,\lm])}{\lm^5} d\lm
\right] & \gray{M^4 = L, \text{ integration by parts}}
\\ & \le  \frac{2L}{\alt \pi^4 T^4}\left[  0 + 4\al' \int_{L/T}^{\Alt} \frac {d\lm} {\lm^4} + 4B \int_{\Alt}^\infty \frac{d\lm}{\lm^5} \right] & \gray{\nu\in \cM_{(0,\al'),\Alt} \text{ and } \nu(\R)\le B}
\\ & 
\le   \frac{1}{L^2 T} + \frac {B L}{\alt (\Alt T)^4}
\le  \frac{2 }{L^2 T}. & \gray{\frac{L^3}{T^3} \le\frac{\al' \Alt^4}{B } }
\end{align*}

Next, by Observation~\ref{obs: comp bound on r},
\[
\text{var}\left(g(s+t)- g(s)\right) = \frac{2}{\al'}(r_g(0)-r_g(t))
\le V_T t^2,
\]
where, by similar considerations,
\begin{align*}
V_T = \frac{2}{\al'}\int_{L/T}^\infty \lm^2 h^2(\lm) d\nu(\lm) &=\frac{2}{\al'}\int_{L/T}^\infty \sinc^4(\tfrac T M \lm) \lm^2 d\nu(\lm)
 \le  2\frac{M^4}{\pi^4 T^4} \int_{L/T}^\infty \frac {d\nu(\lm)}{\lm^2}
\\ & \le \frac{2L}{\al'\pi^4 T^4}\left[ 2\alpha' \int_{L/T}^{\Alt} \frac{d\lm}{\lm^2}+2B \int_{\Alt}^\infty \frac{d\lm}{\lm^3} \right] 
& \gray{\nu\in \cM_{(0,\al'),\Alt} \text{ 
 and   }\nu(\R)\le B}
 \\ &
\le  \frac 1 {T^3} + 
\frac {B}{\alt D^2 T^3}\cdot \frac {L}{T}\le \frac{2}{T^3} \le 1, & \gray{\frac{L}{T}\le \frac{\al' \Alt^2}{B}}
\end{align*}
for $T\ge2$.
Using Lemma~\ref{lem: win ball}, 
we obtain the existence of $c>4$ such that for any 
$T\ge c\big(\frac{1}{\eta}+1\big)$ 
we have
\[
\P\left(\sup_{I_T} |f_{h^2(\nu-\nu_{L/T})}| > \eta\sqrt{\alt}\right) \le 2e^{-\frac{\eta^2 L^2 T}{16} }.
\]
Then 
for $L\ge \frac{4\sqrt{2R}}{\eta}$ and $T\ge \max\{c\big(\frac{1}{\eta}+1\big),\frac{1}{A}\}$
we have
\[
\P\left(\sup_{I_T} |f_{h^2(\nu-\nu_{L/T})}| > \eta\sqrt{\al'}\right) \le 2e^{-2R T} \le \P\Big( f_{\mu+h^2 \nu_{\frac L T}} > \ell  \text{  on  } I_T \Big).
\]
Combining this with~\eqref{eq: smoothing with h} and~\eqref{eq: two parts} we obtain
\begin{align*}
\P\Big(f_{\mu+\nu} > \ell \text{  on  } [0,T]\Big) \le 2 \P\left( f_{\mu+h^2 \nu_{\frac L T}} > \ell-\eta\sqrt{\al'}  \text{  on  } I_T\right),
\end{align*}
which yields, by a further application of Lemma~\ref{lem: level cont}, that for $T\ge \max\{c\big(\frac{1}{\eta}+1\big),\frac{1}{A}\}$ we have
\begin{equation}\label{eq: almost there}
\ttt{\mu+\nu}  
\ge
\big(1-\tfrac 2 {L^{1/4}}\big)
\tttlt{\mu+h^2\nu_{\frac L T}}{\ell}{T\big(1-\tfrac 2{L^{1/4}}\big)}-\tfrac{2\log 2}{T}-c_0 \eta\sqrt{\alt},
\end{equation}
where $c_0>0$ is a constant. Notice that 
\[
(1-h^2) \nu_{\frac L T}(\R) \le \nu([-\tfrac L T,\tfrac L T]) \le 2 \al' \tfrac L T \le \tfrac 1 T,
\]
assuming $L \ge \frac 1 {2\alt}$. If $\eta< 8\sqrt{2R} \alt$, then this holds whenever $L\ge \frac{4\sqrt{2R}}{\eta}$.

Therefore, by part~\ref{item: cont a} of Lemma~\ref{lem: meas cont}, we may continue from~\eqref{eq: almost there} to conclude that, for $\eta<8\sqrt{2R}\alt$, $T\ge \max\left\{c\Big(\frac{1}{\eta}+1\Big),\tfrac 1 A\right\}$ and $L\in \left[ \frac{4\sqrt{2R}}{\eta}, \min \left\{ \left(\tfrac {\alt D^4}{B}\right)^{1/3},
\tfrac{\alt D^2}{B}
\right\}T\right]$, we have
\[
\ttt{\mu+\nu} \ge
\big(1-\tfrac 2 {L^{1/4}}\big)
\tttlt{\mu+\nu_{\frac L T}}{\ell}{T\big(1-\tfrac 2 {L^{1/4}}\big)}
- C_{\eta,\al'},
\]
where $\lim_{\eta\to 0}C_{\eta,\al'}=0$. The desired formulation follows on noting that $\Alt\le 1$ and $\alt\le B$.

\end{proof}

\subsubsection{Proof of Part \ref{item: cont b}}
 By Lemma~\ref{lem: FF LB}, there exists $R\in (0,\infty)$ such that, for any measure $\rho\in  \cM_{\al, A} \cap \cL_{\beta,B}$ we have
\begin{equation}\label{eq: def of M}
\forall T\ge \max\{\tfrac 1 A    ,1\}: \quad \P\left(\inf_{[0,T]}f_{\rho}>\ell-1\right) \ge e^{-RT}.
\end{equation}

Recall that $\nu\in \cM_{(0,\al'),\Alt}$. 
Then, applying Proposition~\ref{prop: truncate} with our given $\alt$ and $\Alt$, $\eta = \sqrt{\ep}$ and $L=\frac{c}{\sqrt{\ep}}$ with a suitable constant $c>0$, we obtain
\begin{equation*}
 \left(1- \xi_\ep\right)
 \tttlt{\mu+\nu_{\frac{L}{T}}}{\ell}{T(1-\xi_\ep) }\le \ttt{\mu+\nu}+C^{(1)}_{\ep,\al'},
\end{equation*}
 for all $\ep<c_1 (\al')^2$ and $T\ge \max\left\{ \frac{c_2}{\sqrt{\ep}}, \frac 1 A\right\}$, with $c_1, c_2>0$ suitably chosen, where
$\xi_\ep:=2\left(\frac{\ep}{c^2 }\right)^{1/8}$
 and
 $\lim\limits_{\ep\to 0} C^{(1)}_{\ep,\al'}=0$.
 
 Since $T(1-\xi_{\ep})\ge \tfrac T 2$ for $\ep<\frac {c^2}{2^{16}}$, invoking~\eqref{eq: def of M}, for all $\ep<\min\{c_1(\alt)^2,\frac {c^2}{2^{16}}\} $ and $T\ge \max\left\{ \frac{c_2}{\sqrt{\ep}}, \frac 1 A\right\}$ we have 
  \[
\tttlt{\mu+\nu_{\frac{L}{T}}}{\ell}{T(1-\xi_\ep)}
  \le 2R,\] so the previous inequality yields
 \begin{equation}\label{eq: 2}
  \tttlt{\mu+\nu_{\frac{L}{T}}}{\ell}{T(1-\xi_\varepsilon) }\le \ttt{\mu+\nu}+C^{(2)}_{\ep,\al'},
\end{equation}
where $\lim\limits_{\ep\to 0} C^{(2)}_{\ep,\al'}=0$. 
Denoting $m= \nu([-\frac L T, \frac L T])$ we observe that
$\cF[ \mu + \nu_{\frac L T} ] \le  \cF[\mu + m \delta_0]$,
so that Slepian's inequality  (Proposition~\ref{prop: slep}) yields
\begin{align}\label{eq: 3}
 \tttlt{\mu+m \delta_0}{\ell}{T(1-\xi_\varepsilon) }\le  \tttlt{\mu+\nu_{\frac{L}{T}}}{\ell}{T(1-\xi_\varepsilon) }.
\end{align}
Proceeding to estimate the LHS of~\eqref{eq: 3}, with 
$I_T= [0,T(1-\xi_\varepsilon)]$
and  $Z\sim N(0,1)$, we have
\begin{align}
\begin{split}
\label{eq: 4}
\P\left(\inf_{I_T}f_{\mu+m\delta_0}>\ell\right)&=\P\left(\inf_{I_T}f_\mu \oplus \sqrt{m} Z>\ell\right)
\\ &\le \P\left(Z>\left( \frac{T}{m}\right)^{1/4}\right) + \P\left(\inf_{I_T}f_\mu>\ell-(mT)^{1/4}\right)
\\& \le \P\left(Z>\sqrt{\frac{T}{\sqrt{2\ep L}}}\right)+\P\left(\inf_{I_T} f_\mu>\ell-(2\ep L)^{1/4}\right),
\end{split}
\end{align}
where the last inequality uses the given assumption that 
$m\le 2\ep \frac{L}{T}$ for all $T\ge \frac{L}{A}=\frac{c}{A\sqrt{\ep}}$.
Since $\ep L=c\sqrt{ \ep}<1/2$, for $\ep<\frac 1 {4c^2}$, so by \eqref{eq: def of M} the second term in the RHS of \eqref{eq: 4} is bounded below by $e^{-RT}$ for $T\ge \max\{\tfrac 1 A,1\}$.
By Lemma~\ref{lem: tail}, the first term in the RHS of \eqref{eq: 4} is bounded above by $e^{-\frac{T}{2\sqrt{2\ep L}}}$, for all $T\ge 4\ge 4\sqrt{2\ep L}$. Thus for 
$T\ge \max\{4,\frac 1 A\}$ and $\ep<   (64c^2 R^4)^{-1}$, the second term is larger than the first, yielding
\[
\P\left(\inf_{I_T}f_{\mu+ m \delta_0}>\ell\right)\le 2 \P\left(\inf_{I_T}f_\mu>\ell-
(2c\sqrt{ \ep})^{\frac 1 4} \right).
\]
By taking $\log$ and dividing by $T(1-\xi_\varepsilon)$ we obtain that, for 
$\ep < c_3\min\{(\alt)^2,1\}$  and 
$T\ge \frac{c_4}{A\sqrt{\ep}}$ (with suitable constants $c_3,c_4>0$), 
\begin{align}\label{eq: 5}
\tttlt{\mu+m \delta_0}{\ell}{T(1-\xi_\ep) }\ge  \tttlt{\mu }{\ell-
(2c\sqrt{ \ep})^{\frac 1 4} }
{T(1-\xi_\ep) }-\tfrac{2\log 2}{T}.
\end{align}
Finally, using Lemma \ref{lem: level cont}, we have for $T\ge \max\{4,\tfrac 1 A, \frac{1}{\sqrt{\ep}}\}$,
\begin{align}\label{eq: 6}
 \tttlt{\mu }{\ell}{T(1-\xi_\ep) }\le  \tttlt{\mu }{\ell-
(2c\sqrt{ \ep})^{\frac 1 4} }{T(1-\xi_\ep) }+C^{(3)}_{\ep},
\end{align}
where $\lim_{\ep\to 0}C^{(2)}_{\ep}=0$.
Combining \eqref{eq: 2}, \eqref{eq: 3}, \eqref{eq: 5} and \eqref{eq: 6} we get
\[
\tttlt{\mu }{\ell}{T(1-\xi_\ep) }\le \tttlt{\mu+\nu }{\ell}{T }+C_{\ep,\alt},
\]
for 
$\ep < c_3\min\{(\alt)^2, 1\}$ and 
$T\ge  \frac{c_5}{A\sqrt{\ep}}$,
where $\lim\limits_{\ep\to 0} C_{\ep,\alt} =0$. The desired conclusion follows.




\section{Ball exponent of singular measures}\label{sec: singular}
In this section we establish Theorem~\ref{thm: nontrivial ball}. We rely on the following proposition.

\begin{prop}\label{prop: singular on Z}
Let $\rho$ be a purely-singular measure supported on $[-\pi,\pi]$.  Then any $\ell>0$ satisfies
\[
\psi_{\rho;1}^\ell=\lim_{T\to\infty} \frac 1{T} \log \P\left( \sup_{j\in [0,T]\cap \Z} |f_\rho(j)| < \ell \right) =  0.
\]
\end{prop}

Using this, we prove Theorem~\ref{thm: nontrivial ball} in Section~\ref{sec: cor nont}.
We discuss a useful approximation method of a GSP with compactly supported spectral measure in Section~\ref{sec: decomp}.
This method is used to prove Proposition~\ref{prop: singular on Z} in Section~\ref{sec: singular on Z}.

\subsection{Proof of Theorem \ref{thm: nontrivial ball}: characterization of vanishing ball exponent}\label{sec: cor nont}

Let $\rho\in\cL$ and $\ell>0$.
Assume first that $\rho$ is purely singular.
For any given $\ep>0 $, we define:
\[
f_{\rho; \ep}(t)=f_\rho\left(\ep \left\lfloor \tfrac{t}{\ep} \right\rfloor\right).
\]
Note that $f_{\rho;\ep}$ is a centered Gaussian process, though it is not stationary.
By Khatri-Sidak's inequality (Proposition~\ref{prop: KS}):
\[
\psi_\rho^\ell(T) \le \psi_{\rho;\ep}^{\ell/2}(T) - 
 {\frac{\lceil T \rceil}{T} } \log \P\left( \sup_{[0,1]} |f_\rho - f_{\rho;\ep}|<\tfrac{\ell}{2}\right).
\]
By letting $T\to\infty$ and using Proposition~\ref{prop: singular on Z} we get
\[
0\le \psi_\rho^\ell \le \psi_{\rho;\ep}^{\ell/2} - \log \P\left( \sup_{[0,1]} |f_\rho - f_{\rho;\ep}|<\tfrac{\ell}{2}\right)
= - \log \P\left( \sup_{[0,1]} |f_\rho - f_{\rho;\ep}|<\tfrac{\ell}{2}\right).
\]
The desired conclusion follows by letting $\ep\to 0$, and noting that $\sup|f_\rho - f_{\rho;\ep}|$ converges to $0$ almost surely (by continuity of sample paths).

Assume now that $\rho$ is not purely singular, that is, $\rho_{ac}\neq 0$.
By \cite[Claim 3.4]{FFN}, there exist $\Lambda=\{\lm_n\}$ of positive density $a$, and a constant $b>0$, such that
$f(\lm_n) \overset{d}{=} b Z_n \oplus g_n$, where $Z_n$ are i.i.d. standard normal random variables and $g_n$ is a Gaussian process on $\Z$.
By Anderson's inequality (Proposition~\ref{prop: and}), this implies
\begin{align*}
 \P \left( \sup_{[0,T]} |f_\rho | <\ell\right)
 &\le \P \left( \max_{\lm \in \Lambda\cap [0,T]} |f_\rho(\lm) | <\ell\right)
 \\ & \le \P\left( \max_{n\in \N \cap [0, \tfrac a 2 T]} |b Z_n | \le \ell \right)
 = \P\left( |Z| \le \frac{\ell}{b} \right)^{\lfloor \tfrac a 2 T\rfloor} \le e^{-C T},
\end{align*}
where $C>0$ depends on $\ell$ and the spectral measure $\rho$. Thus
$\psi_\rho^\ell \ge C>0$, as required.


\subsection{A spectral approximation method}\label{sec: decomp}
The following lemma presents an approximation method for GSPs with compactly supported spectral measure.
\begin{lem}\label{lem: interval decomp}
Let $\rho$ be a spectral measure supported on $[-D, D]$. For $\enn\in\N$ we denote the intervals
$I^{\enn}_{1} = [0,\frac{D}{\enn}]$, $I_{-1}^{\enn} = [-\frac {D}{\enn},0)$  and $I^{\enn}_{\pm j} =\pm ((j-1)\frac{D}{\enn},j\frac{D}{\enn}]$ for $j\in \{2,\dots, \enn\}$, as well as
\begin{equation}\label{eq: Cj and Sj}
C_j(t) =\frac 1{\rho(I_{j})}  \int_{I_{j}} \cos(\lm t) d\rho(\lm),
\quad
S_j(t) = \frac 1{\rho(I_{j})}  \int_{I_{j}} \sin(\lm t) d\rho(\lm).
\end{equation}
Then,
\begin{equation}\label{eq: def R}
f_\rho(t) \overset{d}{=}
 \sum_{j=1}^{ \enn  }    \sqrt{\rho(I_j\cup I_{-j})} \Big(\zeta_j C_j(t)\oplus \eta_j S_j(t)\Big)  \oplus R_{\enn}(t),
\end{equation}
where $\{\zeta_j\}_{j=1}^{\enn}\cup\{\eta_j\}_{j=1}^{\enn}$ are i.i.d. $\mathcal{N}(0,1)$-distributed random variables, and $R_{\enn}(t)$ is a Gaussian process independent of them for which
\begin{equation}\label{eq: goal var}
\sup_{t\in [0,T]} \var(R_{\enn}(t)  ) \le  \frac 1 2\left(\frac{DT}{\enn}\right)^2\rho([-D,D])
\end{equation}
and for any $h\in\R$,
\begin{equation}\label{eq: goal dif}
\var(R_{\enn}(t)-R_{\enn}(t+h)  ) \le  D^2 \rho([-D,D])h^2.
\end{equation}
\end{lem}

\begin{proof}

Notice that
\begin{equation*}
C_j(t) =\FF\left[\tfrac 1{2\rho(I_j)} \left(
\ind_{I_j}+\ind_{I_{-j}}\right) \ d\rho  \right](t), \quad
S_j(t) =\FF\left[\tfrac 1{2\rho(I_j)} \left(
\ind_{I_j}-\ind_{I_{-j}}\right) \ d\rho  \right](t),
\end{equation*}
and
\[
\norm{\frac 1{2\rho(I_j)} (\ind_{I_j} \pm \ind_{I_{-j}} ) }_{L^2_\rho} =\frac {1}{2\rho(I_j)} \sqrt{ \int (\ind_{I_j}^2 + \ind_{-I_j}^2) d\rho }  =\frac 1 {\sqrt{2\rho(I_j)}}.
\]
Since $\{\ind_{I_j} \pm \ind_{-I_j} \}_{j=0}^{  \enn }$ is an orthogonal system in $\cL^2_\rho$,
by the Hilbert decomposition (Lemma \ref{lem: hilbert}) we have the representation~\eqref{eq: def R}.
Let $t\in [0,T]$. By \eqref{eq: def R},
\begin{align}\label{eq: var R}
\var \big(R_{\enn}(t)) & = \rho([-D,D] ) -  \notag
\var\left( \sum_{j=0}^{ \enn }    \sqrt{\rho(I_j\cup I_{-j})} \Big(\zeta_j C_j(t)\oplus \eta_j S_j(t)\Big)  \right)
\\ & =\sum_{j=1}^{\enn} \rho(I_j\cup I_{-j}) \Big(1- (C_j^2(t) + S_j^2(t)) \Big).
\end{align}

Using \eqref{eq: Cj and Sj} we compute, for each $j\in [n]$, that
\begin{align*}
C_j^2(t) + S_j^2(t) &= \frac 1{\rho(I_j)^2}  \left[ \left( \int_{I_j} \cos(\lm t)d\rho(\lm) \right)^2
+\left( \int_{I_j} \sin (\lm t)d\rho(\lm) \right)^2 \right]
\\ & = \frac 1{\rho(I_j)^2} \int_{I_j} \int_{I_j}\Big( \cos(\lm_1 t) \cos(\lm_2 t) +\sin(\lm_1 t)\sin(\lm_2 t)\Big)  d\rho(\lm_1)d\rho(\lm_2)
\\ & = \frac 1{\rho(I_j)^2} \int_{I_j}\int_{I_j} \cos((\lm_1-\lm_2)t) d\rho(\lm_1)d\rho(\lm_2)
\\ & \ge \frac 1{\rho(I_j)^2} \int_{I_j}\int_{I_j} \left(1- \tfrac{1}{2}|(\lm_1-\lm_2)t|^2\right) d\rho(\lm_1)d\rho(\lm_2)  \ge 1-\frac {D^2T^2}{2n^2},
\end{align*}
where in the last step we used that $|(\lm_1-\lm_2)t|\le \frac{D}{\enn} T$ for any $\lm_1,\lm_2\in I_j^{n}$ and $t\in [0,T]$. From \eqref{eq: var R} we now obtain:
\begin{equation*}
\var \big(R_{\enn})(t)  \le \frac 1 2\left(\frac{DT}{\enn}\right)^2 \sum_{j=1}^{\enn} \rho(I_j \cup I_{-j})   =  \frac 1 2\left(\frac{DT}{\enn}\right)^2\rho([-D,D]) , \quad \forall t\in [0,T],
\end{equation*}
thus verifying \eqref{eq: goal var}. Moreover, by the independent decomposition~\eqref{eq: def R} and Observation~\ref{obs: comp bound on r},
\begin{align*}
\var\left( R_{\enn}(t) - R_{\enn}(t+h) \right)
\le \var\left( f_\rho(t) -f_\rho(t+h) \right)
\le D^2  \rho([-D,D]) h^2,
\end{align*}
which establishes \eqref{eq: goal dif}.
\end{proof}

\subsection{Proof of Proposition~\ref{prop: singular on Z}}\label{sec: singular on Z}

Assume without loss of generality that $\rho([-\pi,\pi])=1$.
Let $\enn \in \N$ and $D=\pi$, and for $|j|\in [n]:= \{1,2,\dots,\enn\}$ define $C_j$ and $S_j$ as in \eqref{eq: Cj and Sj}. Then by Jensen's inequality we have
\begin{equation}\label{eq: C2+S2}
\sup_{t\in\R}  \left\{|C_j(t)|^2+ |S_j(t)|^2\right\}  \le 1.
\end{equation}
By Lemma~\ref{lem: interval decomp}, the decomposition \eqref{eq: def R} holds, with bounds as in \eqref{eq: goal var} and \eqref{eq: goal dif}.

Fix $\ep>0$ and partition the indices in $[n]$ into
\begin{equation}\label{eq: AB}
 \cA_{\enn,\ep}  = \left|\left\{j\in [\enn] : \: \rho(I_j\cup I_{-j})\ge  \frac {\pi \ep}{\enn}\right\}\right|, \quad \cB_{\enn,\ep}  = [\enn]  \setminus \cA_{\enn,\ep} ,
\end{equation}
defining accordingly the functions
\begin{align*}
A_{\enn,\ep} (t) &= \sum_{j\in \cA_{\enn,\ep} }  \sqrt{\rho(I_j\cup I_{-j})} \Big(\zeta_j C_j(t)\oplus \eta_j S_j(t)\Big),\\
B_{\enn,\ep} (t) &=\sum_{j\in \cB_{\enn,\ep} }   \sqrt{\rho(I_j\cup I_{-j})} \Big(\zeta_j C_j(t)\oplus \eta_j S_j(t)\Big) ,
\end{align*}
so that \eqref{eq: def R} becomes
\begin{equation}\label{eq: ABR}
f_\rho \overset{d}{=} A_{\enn,\ep}  \oplus B_{\enn,\ep}  \oplus R_\enn .
\end{equation}

Observe that
\begin{equation*}
\P\Big( |f_\rho| < \ell \: \text{ on } [0,T] \cap \Z\Big) \ge
\P\Big(\sup_{[0,T]\cap \Z}\big| A_{\enn,\ep}  \big| <\tfrac  \ell 3\Big)
\P\Big(\sup_{[0,T]\cap \Z} \big| B_{\enn,\ep}   \big| <\tfrac  \ell 3 \Big)
\P\Big(\sup_{[0,T]\cap \Z} \big| R_{\enn}   \big| <\tfrac  \ell 3 \Big).
\end{equation*}

Given $T\in\N$, we carry out this decomposition with $\enn = mT$ (the parameter $m\in\N$ will be chosen later).
Proposition~\ref{prop: singular on Z} reduces to the following claims.

\begin{clm}\label{clm: A}
For any fixed $m \in \N$ and $\ep>0$,
\[ \lim_{T\to \infty} \frac 1 T \log \P\left(\sup_{[0,T]} \big| A_{mT,\ep}  \big| <\tfrac \ell 3 \right) =0.\]
\end{clm}

\begin{clm}\label{clm: B}
For any $m\in\N$,

\begin{equation*}
\lim_{\ep\to 0} \lim_{T\to\infty}\frac 1 T \log \P\left(\sup_{[0,T]} \big| B_{mT,\ep}   \big| <\tfrac \ell 3 \right) =0.
\end{equation*}
\end{clm}

\begin{clm}\label{clm: R}

\begin{equation*}
\lim_{m\to\infty} \lim_{T\to\infty}  \frac 1 T \log \P\left(\sup_{[0,T]} \big| R_{mT}  \big| <\tfrac \ell 3 \right) =0.
\end{equation*}
\end{clm}

We turn to verify the claims.

\begin{proof}[Proof of Claim \ref{clm: A}]
Fix $m\in\N$ and $\ep>0$. 
We begin by showing that
\begin{equation}\label{eq:singular is mostly rare}
\lim\limits_{T\to\infty}\frac{|\cA_{mT,\ep}|}{T}=0.
\end{equation}
To this end we use the following classical fact (see \cite[Chapter VII.6, Thm 2]{Shi}): given a filtration $\cF_T \nearrow \cF$, a measure $\rho$
is purely singular with respect to another measure $\mu$ if and only if
\begin{equation*}
\lim_{T\to\infty} \frac{d\rho |_{\cF_T}}{d\mu |_{\cF_T}} \overset{\mu\text{-a.s.}}=0.
\end{equation*}
Since almost sure convergence implies convergence in probability, this implies that if $\rho$ is singular w.r.t. $\mu$ then for any $\ep>0$ we have
\begin{equation}\label{eq: shir p}
\lim_{T\to\infty} \mu \left(\frac{d\rho |_{\cF_T}}{d\mu |_{\cF_T}}> \ep  \right) =0.
\end{equation}
We apply this, taking $\rho$ to be the given spectral measure, $\mu$ -- the Lebesgue measure on~$[-\pi,\pi]$ and $\cF_T=\sigma(\{I^{mT}_{j}\}_{1\le |j|\le mT})$. Observing that
\[
\frac{d\rho |_{\cF_T}}{d\mu |_{\cF_T}} =\sum_{1\le |j|\le mT}  \frac{\rho(I^{mT}_j) }{\mu(I^{mT}_j)}\ind_{I_j} =
\sum_{1\le |j|\le mT} \frac{ \rho(I^{mT}_j) }{ \pi /(mT) }\ind_{I_j},
\]
we obtain from
\eqref{eq: shir p}:
\begin{equation}\label{eq: FT}
\mu \left( \sum_{j} \frac{ \rho(I^{mT}_j) }{ \pi /(mT) }\ind_{I_j} > \ep  \right)=\frac{\pi \left|\left\{j\in\{  \pm1,\dots,  \pm mT\} \ :\ \rho(I_j)\ge \frac {\pi}{mT} \ep \right\}\right|}{mT}\overset{T\to\infty}\longrightarrow 0.
\end{equation}
Combining this with \eqref{eq: AB}, and recalling that $\rho(I_j)\ge \rho(I_{-j})$ for $j\ge 1$ (equality holds for $j\ne 1$) we obtain that
\[
\lim\limits_{T\to\infty}\frac{|\cA_{mT,\ep}|}{T} =  \lim\limits_{T\to\infty}\frac{\left|\left\{j\in [n] \ :\  \rho(I_j)\ge \frac{\pi}{mT} \ep \right\}\right|}{T} =  0,
\]
thus \eqref{eq:singular is mostly rare} is established.

Denoting $d=2|\cA_{mT,\ep}|$, we recall that $A_{mT,\ep}(t)= \langle  u(t), \zeta \rangle$ is an inner product in $\R^d$ between
$$u(t) = \Big(\sqrt{\rho(I_j \cup I_{-j}) } C_j(t),\sqrt{\rho(I_j \cup I_{-j}) } S_j(t)  \Big)_{j\in \cA_{mT,\ep}}$$
and a standard $d$-dimensional multi-normal random vector $\zeta\sim \g_d$ .
Note that, by \eqref{eq: C2+S2},
\[
\norm{ u (t)}^2 = \sum_{j\in\cA_{mT,\ep}} \rho(I_j\cup I_{-j} ) (C_j^2(t) +S_j^2(t) ) \le \sum_{0\le j< mT} \rho(I_j \cup I_{-j}) =\rho([-\pi,\pi]) =1.
\]
Hence,
 \begin{align*}
\P\left( \bigcap_{k=1}^{T} \left\{ |A_{mT,\ep}(k)| < \tfrac \ell 3\right\}\right)
&= \g_{d} \left(  \bigcap_{k=1}^{T}
\left\{ \zeta \in \R^d : \ \left| \left\langle \zeta,  u(k) \right\rangle \right| \le \tfrac \ell 3 \right\} \right) \notag
\\ & \ge  \g_{d} \left(  \bigcap_{k=1}^{T }
\left\{ \zeta \in \R^d : \ \left| \left\langle  \zeta, \tfrac {  u(k)}{ \norm{u(k)} } \right\rangle \right| \le \tfrac \ell 3 \right\} \right)
& \textcolor{dark gray}{ \norm{ u(k) } \le 1  }
\\ & \ge  \left(\frac \ell 3\right)^d
 \g_{d} \left(  \bigcap_{k=1}^{T}
\left\{ \zeta \in \R^d : \ \left| \left\langle  \zeta, \tfrac {  u(k)}{ \norm{u(k)} } \right\rangle \right| \le 1 \right\} \right)
&\textcolor{dark gray}{\text{by Obs. \ref{obs: stretch}} }
  \\ & \ge
\left(  \frac{\kappa \ell }{3 \sqrt{1+2\log \frac {T}{d}}}\right)^d,
& \textcolor{dark gray}{\text{by Prop. \ref{prop: slabs}} }
\end{align*}
where $\kappa$ is a universal constant.
Thus we obtain
\[\lim_{T \to \infty} \frac 1 T \log \P\left(\bigcap_{k=1}^T \left\{\big| A_{mT,\ep}(k) \big| <\tfrac \ell 3 \right\} \right)
\ge \lim_{T \to \infty} \frac { 2|\cA_{mT,\ep}|} {T } \log \left(\frac {c\ell}{\sqrt{\log \left( \frac T {2|\cA_{mT,\ep}| } \right) }} \right) = 0,
\]
where
$c>0$ is a universal constant, and the last equality follows from \eqref{eq:singular is mostly rare}.
The claim follows.
\end{proof}

\begin{proof}[Proof of Claim \ref{clm: B}]
Fix $m\in \N$.
Using \eqref{eq: C2+S2} and \eqref{eq: AB}, we have for any $0\le t\le T$ and $\ep>0$,
\begin{align*}
\var B_{mT,\ep}(t)  &= \sum_{j\in\cB_{mT,\ep}} \rho(I_j \cup I_{-j}) \left( C_j^2(t)+S_j^2(t)\right)
\\ & \le |\cB_{mT,\ep}|\cdot 2\max_{j\in\cB_{mT,\ep}} \rho(I_j) \sup_{t\in [0,T]} \left(C_j^2(t) +S_j^2(t)\right)
\\ & \le  mT \cdot \frac{2\pi \ep}{mT} \cdot 1 = 2\pi \ep.
\end{align*}

Using Khatri-Sidak's inequality (Proposition \ref{prop: KS}) and Lemma~\ref{lem: tail}, we have:
\begin{align*}
-\frac 1 T \log \P\left(\sup_{t\in (0,T]\cap\N} \big| B_{mT,\ep} (t) \big| <\tfrac \ell 3 \right)
&\le -\frac 1 T \sum_{t=1}^T \log \P\left(\big|B_{mT,\ep}(t)  \big| <\tfrac \ell 3 \right)
\\ &  \le  -\log \P\left( \sqrt{2\pi \ep} |\cN(0,1) | < \tfrac{\ell}{3} \right)
\\ & \le -\log \left( 1- 2 e^{-\frac{\ell^2}{36 \pi\ep}}\right) \overset{\ep \to 0}\longrightarrow 0.
\qedhere
\end{align*}
\end{proof}

\begin{proof}[Proof of Claim \ref{clm: R}]
By \eqref{eq: goal var},
\begin{equation*}
\var \big(R_{mT})(t) \le  \frac{\pi^2}{2m^2}, \quad \forall t\in [0,T].
\end{equation*}
Now using Khatri-Sidak's inequality (Proposition \ref{prop: KS}) and a tail estimate (Lemma~\ref{lem: tail}), we have:
\begin{align*}
-\frac 1 T \log \P\left(\sup_{t\in (0,T]\cap\N} \big| R_{mT} (t) \big| <\tfrac \delta 3 \right)
 &\le  -\frac 1 T \sum_{t=1}^T \log \P\left(\big| R_{mT}(t)  \big| <\tfrac \ell 3 \right)
\\ &  \le  -\log \P\left( \tfrac{\pi}{\sqrt{2} m} |\cN(0,1) | < \tfrac{\ell}{3} \right)
\\ & \le   -\log \left( 1- 2 e^{-\frac{\ell^2}{9 \pi^2}m^2}\right) \overset{m\to\infty}\longrightarrow 0.
\qedhere
\end{align*}
\end{proof}

\section{Existence and continuity of the persistence exponent}\label{sec: main thm}

In this section we prove Theorems~\ref{thm: main exist} and~\ref{thm: cont}.
Our method is to approximate the spectral measure by smooth spectral measures, for which it is easier to prove the existence of the exponent, and then use the comparison lemmata (proved in Section~\ref{sec: key lemmas}) in order to retrieve existence for the original measure.
To make this idea concrete, we formulate three auxiliary results.
The first provides existence of the persistence exponent for smooth compactly supported spectral densities. 

\begin{prop}\label{prop: summable}
Let $\rho$ be an absolutely continuous spectral measure with compactly supported density which is twice differentiable on $\R$.
Then $\theta_\rho^\ell := \lim_{T\to\infty} \ttt{\rho}$ exists in $(0,\infty]$.
\end{prop}

The second result states that persistence exponents are close if the spectral measures are equal near the origin and close in total variation. 

\begin{prop}\label{prop: equal in}
Let $c>0$ be the constant from Lemma~\ref{lem: meas cont}-\eqref{item: cont b}.  
Fix $\alt>0,\Alt\in (0,1)$ and let $\mu, \nu \in \mathcal{M}_{(0,\al'),\Alt}\cap\mathcal{M}_{\alpha,A}\cap \mathcal{L}_{\beta,B}$. Let $\ep\in (0,1)$ and $T\ge \max\{4, \frac 1 A, \frac {1}{\ep}\}$ such that $d_{TV}(\mu,\nu)<\ep$ and $\mu|_{[-\frac{c}{T\sqrt{\ep}},\frac{c}{T\sqrt{\ep}}]}=\nu|_{[-\frac{c}{T\sqrt{\ep}},\frac{c}{T\sqrt{\ep}}]}$. Then
\[
\tttlt{\nu}{\ell}{ T(1-c\ep^{1/8}) } \le   \ttt{\mu}  +C_{\ep,\alt},
\]
where $\lim\limits_{\ep\to 0 }C_{\ep,\alt}=0$.
\end{prop}

By the last result, persistence exponents are close also when the spectral measures are equal away from the origin and close near the origin.
\begin{prop}\label{prop: equal out}
There exists a constant $c$ such that for any $L\ge 1$, $\ep\in (0,1)$, $m>0$ and  {$T\ge c\left(L^2+m^{-1}+\ep^{-2}+1\right)$}, 
the following holds:
Suppose that $\mu, \nu\in \cM_{\al, 1/T}\cap \cL_{\beta,B}$ are such that
$ \mu|_{\R\setminus [-\frac L T,\frac L T] } = \nu|_{\R\setminus [-\frac L T,\frac L T] }$ and
$\forall \lm \in (0,\tfrac L T):  \ |\mu(-\lm,\lm) - \nu(-\lm,\lm)|\le 2 \ep \lm $, while also
$\nu\left(\left(-\frac L T, \frac L T\right)\right)\le m \tfrac L T$ and $\mu\left(\left(-\frac L T, \frac L T\right)\right)\le m \tfrac L T$.
Then
\[
|\ttt{\mu} - \ttt{\nu}| < C_{\ep},
\]
where $C_{\ep}=C_\ep(L,m)$ and
$\lim\limits_{\ep\to 0 } C_{\ep} =0$.
\end{prop}

We proceed as follows. First we prove Theorems~\ref{thm: main exist} and~\ref{thm: cont} in Sections \ref{sec: main proof} and~\ref{sec: thm cont},
respectively. Then we present the proofs of the auxiliary results (Propositions~\ref{prop: summable}, \ref{prop: equal in} and~\ref{prop: equal out}) in Section~\ref{sec: aux}.
To show the tightness of our conditions, we provide in Section~\ref{sec: exist counter} a counter-example to the existence of a persistence exponent.

\subsection{Proof of Theorem \ref{thm: main exist}: existence of the persistence exponent}
\label{sec: main proof}

Let $\rho\in \cM\cap \cL$. 
If $\rho'(0)=\infty$, then it follows from Proposition~\ref{prop: low explode}
that $\theta_\rho^\ell = 0$ for all $\ell\in\R$.
Thus we assume $\rho'(0)<\infty$. Consequently, $\rho\in \cM_{(\al,\al'),A}\cap \cL_{\beta,B}$ for some $\al,\al',A,\beta,B>0$ which are fixed throughout the proof.

Let $c>0$ be as in the statement of 
Lemma~\ref{lem: meas cont}-\eqref{item: cont b}. Given $\ep>0$, denote
 {\[
L(\ep) = \frac{c}{\sqrt\ep}, \quad \eta(\ep) = c\ep^{1/8}.
\]}
For a given $T>0$,
we approximate the spectral measure in several steps by smoother and smoother measures, without altering the persistence exponent $\theta_\rho^\ell(T)$ significantly. We treat separately the measure near the origin, i.e., in the interval $[-\tfrac L T, \tfrac L T]$, and the measure away from the origin, on the remainder of $\R$.

\bigskip
\textbf{Step 1: discarding the singular part away from the origin.}
We write $\rho = \rho_{\ac} + \rho_{\sing}$ where $\rho_{\ac}$ is absolutely continuous and $\rho_{\sing}$ is purely singular.
For a given $T>0$, set
\[
\g_T := \rho_{\sing}\big|_{\R\setminus [-\frac{L}{T},\frac{L}{T}]},
\]
and define
\[
\mu_T  := \rho - \g_T = \rho\big|_{[-\frac{L}{T},\frac{L}{T}]} + \rho_{ac}\big|_{\R \setminus [-\frac{L}{T},\frac{L}{T}]}.
\]
Let us show that upper and lower limits of $\ttt{\rho}$ and $\ttt{\mu_T}$ are close.
We apply 
Part~\ref{item: cont b} of Lemma~\ref{lem: meas cont}
with $\mu=\mu_T$ and $\nu = \g_T$
to obtain 
that for all  {$T\ge \frac{c}{A\sqrt{\ep}}$} we have
\begin{equation}\label{eq: no sing}
\theta_{ \mu_{T} }^\ell \left( (1-\eta) T \right)
\le
\theta_{ \rho }^\ell \left( T \right) + C^{(1)}_\ep,
\end{equation}
where $\lim_{\ep\to 0} C^{(1)}_\ep =0$.
On the other hand, for $T\ge {\max\{4,\frac 1 A,\frac 1 \ep\}}$ we have
\begin{align*}
\ttt{\rho}
&\le \theta_{\rho}^{\ell-\ep}(T) + C^{(2)}_\ep
& \textcolor{dark gray}{ \text{by Lemma~\ref{lem: level cont}}}
\notag
\\
&\le \theta_{\mu_{T}}^{\ell}(T) + \psi_{\g_T}^{\ep}(T)+ C^{(2)}_\ep
& \textcolor{dark gray}{ \text{by Lemma~\ref{lem: ball-level}(b)}}
\notag
\\
&\le \theta_{\mu_{T}}^{\ell}(T) + \psi_{\rho_{\sing}}^{\ep}(T)+ C^{(2)}_\ep,
& \textcolor{dark gray}{ \text{by Lem.~\ref{lem: spec decomp} and Prop.~\ref{prop: and}, $\g_T\le \rho_{\sing}$}}
\end{align*}
where $\lim_{\ep\to 0} C^{(2)}_\ep =0$.
By Theorem~\ref{thm: nontrivial ball} we have $\lim\limits_{T\to\infty} \psi^\ep_{\rho_{\sing}}(T) = 0$, so that
\[
\limsup_{T\to\infty } \ttt{\rho} \le \limsup_{T\to\infty}\ttt{\mu_T} +C_\ep^{(2)}.
\]
Together with \eqref{eq: no sing} we conclude that 
\begin{equation}\label{eq: step1}
\liminf_{T\to\infty} \tttlt{\mu_T}{\ell}{ {(1-\eta)}T} - C_\ep^{(1)} \le \liminf_{T\to\infty} \ttt{\rho}\le 
\limsup_{T\to\infty} \ttt{\rho} \le \limsup_{T\to\infty}\ttt{\mu_T} +C_\ep^{(2)}.
\end{equation}

\bigskip
\textbf{Step 2: smooth approximation away from the origin.}
We shall employ the following approximation claim,
which could be proved by standard analysis arguments.
\begin{clm}\label{clm: smooth approx}
Let $\rho\in \cM\cap\cM_{(\al,\al'),A} \cap \cL_{\beta,B}$, and let $g\in L^1(\R)$ be the density of~$\rho_{\ac}$. Let $\ep>0$ be given.
Then there exists an absolutely continuous measure $\nu\in \cM\cap\cM_{(\al,\al'),\ep} \cap \cL_{\beta,B}$ with smooth and compactly supported density $h\in C^\infty_0(\R)$ satisfying
\[
\int_\R  |h - g| < \ep, \quad \text{and} \quad  \: \forall \lm\in (-\ep,\ep): \:\: h(\lm) =\rho'(0).
\]
\end{clm}

Let $\nu$ be the measure obtained by applying Claim~\ref{clm: smooth approx} with our given $\rho$ and $\ep$.
For any given $T>0$, define the measure $\si_T$ which approximates the measure $\rho$ away from the origin by the smooth measure $\nu$:
\[
\sigma_{T} := \rho|_{[-\frac {L} T, \frac {L} T]} + \nu|_{\R \setminus [-\frac {L} T, \frac {L} T]}.
\]
Note that $d_{\TV}(\si_{T},\mu_{T})  = \frac 1 2 \int_{\R \setminus [-\frac {L} T, \frac {L} T]}|h - g| <\frac {\ep}{2}$.
By Proposition~\ref{prop: equal in} (applied with $A=\tfrac L T$),  {there exists $\ep_0$ such that
if $\ep<\ep_0$ and $T>T_3(\ep)$}, then
\begin{align*}
\tttlt{\si_{T}}{\ell}{ T(1-\eta) } \le   \ttt{\mu_{T}}  +C^{(3)}_\ep,  \quad \tttlt{\mu_{T}}{\ell}{ T } \le   \tttlt{\si_{T}}{\ell}{\tfrac T{1-\eta} }  +C^{(3)}_\ep,   
\end{align*}
where $\lim_{\ep\to 0} C^{(3)}_\ep =0$. We conclude that
\begin{gather}
\begin{aligned}
\liminf_{T\to\infty} \tttlt{\si_T}{\ell}{(1-\eta)T}-C^{(3)}_\ep \le \liminf_{T\to\infty} \ttt{\mu_T}\le 
 \limsup_{T\to\infty} \ttt{\mu_T}\le \limsup_{T\to\infty} \tttlt{\si_T}{\ell}{\tfrac T{1-\eta}}+C^{(3)}_\ep.
\end{aligned}\label{eq: step2}
\end{gather}

\bigskip
\textbf{Step 3: smooth approximation near the origin.}
Next we verify the conditions of Proposition~\ref{prop: equal out}, which will allow us to compare $\ttt{\si_T}$ and $\ttt{\nu}$. First we note that 
${\si_T}|_{\R\setminus [-\frac L T,\frac L T] } = \nu|_{\R\setminus [-\frac L T,\frac L T] }$.
Recalling that on the interval $[-\frac L T, \frac L T]$ we have 
$\si_T = \rho$
and $d\nu = \rho'(0) d\lm$, 
we obtain that  $\si_T, \nu \in \cM_{(\al,\al'), \frac L T}\cap \cL_{\beta,B}$.
Moreover,  {for $T>\frac{L}{\ep}$} 
we have
\[
\sup_{\lm\in\left(0,\frac {L}{T}\right)} \frac{ \left| \si_{T}(-\lm,\lm) - \nu(-\lm,\lm)\right|}{2\lm} =
\sup_{\lm\in\left(0,\frac {L}{T}\right)} \left| \frac{\rho(-\lm,\lm)}{2\lm} - \rho'(0) \right|
\le \sup_{\lm\in(0,\ep)} \left| \frac{\rho(-\lm,\lm)}{2\lm} - \rho'(0) \right| =: \delta(\ep),
\]
where $\lim_{\ep\to 0}\delta(\ep) =0$. 
We may thus apply Proposition~\ref{prop: equal out} with $m=2\alpha'$ to get that, for $T>T_4(\ep,\al')$,
\begin{equation}\label{eq: step3}
 \left|  \theta_{\si_{T} }^\ell\left((1-\eta)T\right) - \theta_{\nu}^\ell\left((1-\eta)T\right)   \right| \le C^{(4)}_\ep,
\quad \left|  \theta_{\si_{T} }^\ell\left(\tfrac{T}{1-\eta}\right) - \theta_{\nu}^\ell\left(\tfrac{T}{1-\eta}\right)   \right| \le C^{(4)}_\ep,
\end{equation}
 where $\lim_{\ep \to 0} C^{(4)}_\ep = 0$ (here $C_\ep^{(4)}$ depends on $\rho$).

\bigskip

\textbf{Step 4: smooth measures have an exponent.}
Since $\nu$ has a smooth and compactly supported density, by Proposition~\ref{prop: summable}, $\lim_{T\to\infty} \ttt{\nu} = \theta^\ell_{\nu}$ exists.
Thus, for $T>T_0(\ep,\al',A)$, we have
\begin{align*}
\limsup_{T\to\infty}\ttt{\rho}-\liminf_{T\to\infty}\ttt{\rho}
&\le \limsup_{T\to\infty}\ttt{\mu_T}-\liminf_{T\to\infty}\tttlt{\mu_T}{\ell}{ {(1-\eta)}T}+ C^{(1)}_\ep+C^{(2)}_\ep
& \textcolor{dark gray}{ \text{by \eqref{eq: step1}} }
\\
&\le \limsup_{T\to\infty}\tttlt{\si_T}{\ell}{\tfrac{T}{1-\eta}}-\liminf_{T\to\infty}\tttlt{\si_T}{\ell}{(1-\eta) {^2}T}
+ \sum_{j=1}^3 C^{(j)}_\ep
& \textcolor{dark gray}{ \text{by \eqref{eq: step2}} }\\
&\le \limsup_{T\to\infty}\tttlt{\nu}{\ell}{\tfrac{T}{1-\eta}}-\liminf_{T\to\infty}\tttlt{\nu}{\ell}{(1-\eta) {^2} T}
+ \sum_{j=1}^4 C^{(j)}_\ep
& \textcolor{dark gray}{ \text{by \eqref{eq: step3}} }\\
& =   \sum_{j=1}^4 C^{(j)}_\ep.&
 \textcolor{dark gray}{ \text{by Prop.~\ref{prop: summable}} }
\end{align*}  
We conclude that $\limsup\limits_{T\to\infty}\ttt{\rho}=\liminf\limits_{T\to\infty}\ttt{\rho}$, as required.

As a by-product of our proof, we have shown the following.

\begin{prop}\label{prop: theta smooth approx}
Let $\rho\in\cM\cap \cL$ and $\ep>0$, and let $\nu$ be the corresponding smooth measure from Claim~\ref{clm: smooth approx}.
Then
\[
\left| \theta^\ell_\rho - \theta^\ell_\nu \right| \le  C_\ep,
\]
where $\lim_{\ep\to 0}C_\ep =0$ and $C_\ep$ depends on $\rho$.
\end{prop}

\subsection{Proof of Theorem~\ref{thm: cont}: Continuity}\label{sec: thm cont}

\subsubsection{Part \ref{item: cont ball}: continuity of the ball exponent}

\textbf{Continuity in $\ell$:}
The function $\ell \mapsto \psi^\ell_\rho$ on $(0,\infty)$ is finite-valued (by Lemma~\ref{lem: lower ball}) and convex (by log-concavity of Gaussian measures).
Hence $\ell \mapsto \psi^\ell_\rho$ is continuous and differentiable almost everywhere. Furthermore,
\[
0 \ge \frac{\partial}{\partial \ell} \psi^\ell_\rho  \ge \psi^{\ell+1}_\rho - \psi^\ell_\rho \ge -\psi^\ell_\rho \ge -C,
\]
where $C=C(\beta,B)$ is the constant from Lemma~\ref{lem: lower ball}.
This shows that $\ell \mapsto \psi^\ell_\rho$ is locally Lipshitz and uniformly continuous in the class $\rho\in \cL_{\beta,B}$.

\smallskip
\noindent
\textbf{Continuity in $\TV$:}
We start with a general claim.

\begin{clm}\label{clm: HJ}
Suppose $\mu$ and $\nu$ are two spectral measures  {in $\cL_{\beta,B}$}.
Then there exists a spectral measure $\g$ such that  $\g\ge \max\{\mu,\nu\}$, and  {$\g-\mu, \g-\nu\in \cL_{\beta,B}$}, and
$
\max\left\{ d_{\TV}(\g,\mu), d_{\TV}(\g,\nu) \right\} \le 2 d_{\TV}(\mu,\nu).
$
Moreover, if $\mu = \nu$ on an interval $I$, then $\g = \nu$ on $I$.
\end{clm}

\begin{proof}
Setting $\si = \mu-\nu$, we note that $\si$ is a finite signed measure. Let $\sigma=\sigma_+-\sigma_-$ be the Hahn-Jordan decomposition. Define $\g := \nu+\si_+$. Clearly, if $\mu=\nu$ on an interval $I$ then $\g=\nu$ on~$I$. Next observe that
$\g\ge \nu$ and $\g = \nu + \si_+ \ge \nu + \si  = \mu$.
 {Also, we have
$$\g-\nu=\sigma_+=(\mu-\nu)_+\le \mu,$$
and so $\g-\mu, \g-\nu$ are both in $\cL_{\beta,B}$.}

Finally, by the Hahn-Jordan theorem,
\begin{align*}
d_{\TV}(\g, \nu) &= d_{\TV}(\si_+,0) \le d_{\TV}(\mu,\nu),\\ \quad
d_{\TV}(\g, \mu) &\le d_{\TV}(\g, \nu) + d_{\TV}(\nu,\mu) \le 2d_{\TV}(\mu,\nu). 
\qedhere
\end{align*}
\end{proof}

Let  {$\mu, \nu \in \cL_{\beta,B}$} such that $d_{\TV}(\mu,\nu)<\ep$.
Let $\g$ be the spectral measure constructed in Claim~\ref{clm: HJ}, where we note that  {$\g-\mu, \g-\nu\in \cL_{\beta,B}$.}
Since $\g \ge \max\{\mu,\nu\}$ 
we have for arbitrary $\delta\in(0,\ell)$ that
\[
\psi_{\mu}^\ell \le \psi_{\gamma}^\ell\le \psi_{\nu}^{\ell-\delta}+\psi_{\gamma-\nu}^{\delta},
\]
where the first inequality is due to Anderson (Proposition~\ref{prop: and}), and the second is Lemma \ref{lem: ball-level}(a).
Applying Khatri-Sidak (Proposition~\ref{prop: KS}) and Lemma \ref{lem: small sup}, there is a choice of $\delta=\delta(\ep)$ with $\lim_{\ep\to 0}\delta(\ep)=0$ and
\[
 \psi_{\g-\nu}^{\delta(\ep)} \le \psi_{\g-\nu}^{\delta(\ep)} (1) \le \ep.
\]
By the uniform continuity of $\ell\mapsto \psi^\ell_\rho$ in the class $\cL_{\beta,B}$ (proven above), it holds that
\[
\psi_{\nu}^{\ell-\delta(\ep)} \le \psi_\nu^\ell + C_\ep,
\]
where $\lim\limits_{\ep\to 0}C_\ep = 0$.
We deduce that $\psi_\mu^\ell \le \psi_\nu^\ell + \widetilde{C}_\ep$, where $\lim_{\ep\to 0} \widetilde{C}_\ep=0$.
Since the roles of $(\mu,\nu)$ are symmetric, we conclude our proof.

\subsubsection{Part \ref{item: cont pers}: continuity of the persistence exponent}

\textbf{Continuity in $\ell$:} is an immediate consequence of Lemma~\ref{lem: level cont}.

\noindent
\textbf{Continuity in $\TV_0$:}
Recall the definitions in~\eqref{eq: classes}, and
let $\mu,\nu\in \cM\cap \cM_{(\al,\al'),A}\cap\cL_{\beta,B}$ be such that
$d_{\TV_0}(\mu,\nu)<\ep$.
 {Let $c>0$ be the constant whose existence is guaranteed by Proposition~\ref{prop: equal in}.
Set $L=
\frac{c}{\sqrt \ep}$, $\eta = c\ep^{1/8}$.} 
For a given $T>0$ define
\[
\sigma = \mu|_{[-\frac L T, \frac L T]} + \nu|_{[-\frac L T, \frac L T]^c},
\]
and note that $\si\in \cM_{\al,\frac L T} \cap \cM_{(0,\al'),A}\cap \cL_{\beta,B}$.
Using Proposition~\ref{prop: equal in} with $\mu$ and $\si$ we conclude that, for  {$T\ge\max\{4,\frac 1 A, \frac 1{\ep}\}$}, we have
\[
\tttlt{\mu}{\ell}{T(1-\eta)} \le \ttt{\si} + C^{(1)}_\ep,
\]
where $C^{(1)}_\ep=C^{(1)}_\ep(\al')$ is such that $\lim\limits_{\ep\to 0}C^{(1)}_\ep=0$.
Since $\mu\in \cM_{(\al,\al'),A}$, 
we have
$\mu((-\frac{L}{T},\frac{L}{T})) < 2\al' \frac {L}{T}$
for all $T\ge \frac L A=\frac{c}{A\sqrt{\ep}}$, 
which also implies $\si( (-\frac{L}{T},\frac{L}{T})) < 2\al'\frac {L}{T}$.
A similar statement holds for $\nu\in\cM_{(\al,\al'),A}$. 
Applying Proposition~\ref{prop: equal out} with $\si$ and $\nu$ we get that, for $T>T_0(\ep, \al')$,
\[
\ttt{\si} \le \ttt{\nu} + C^{(2)}_\ep
\]
where $C^{(2)}_\ep=C^{(2)}_\ep(\al')$ satisfies
$\lim\limits_{\ep\to 0} C^{(2)}_\ep=0$.
Combining the last two displayed formulas, we get:
\[
\tttlt{\mu}{\ell}{T(1-\eta)}  \le \ttt{\nu} + C^{(1)}_\ep+ C^{(2)}_\ep,
\]
 {for $T\ge T_1(\ep,\alt,A)$.}
By Theorem \ref{thm: main exist} we may take the limit as $T\to \infty$ to obtain
\[
\theta_\mu^\ell \le \theta_\nu^\ell + C_\ep,
\]
 {where $C_\ep=C_\ep(\al')$ and $\lim_{\ep\to 0}C_\ep =0$.}
Since the roles of $(\mu,\nu)$ are symmetric, the uniform continuity of $\rho\mapsto\theta_\rho^\ell$ in 
$\cM\cap \cM_{(\al,\al'),A}\cap \cL_{\beta,B}$ follows.

\subsection{Proofs of the auxiliary propositions}\label{sec: aux}
\subsubsection{Proof of Proposition \ref{prop: summable}}\label{sec: summable}
We begin by stating a comparison lemma for persistence probabilities of ``approximately stationary'' Gaussian processes, which implies Proposition~\ref{prop: summable} and is useful in the proof of Theorem~\ref{thm: sample}.

\begin{lem}\label{lem:sampling_fast}
Let $\eta,M>1$ and $c>0$. Suppose that $\{X(t)\}_{t\ge 0}$ is a centered Gaussian process with $\E X(t)^2=1$,
such that
\[
 \E X(s)X(s+t) \le c|t|^{-\eta}, \quad \text{for all } s,t\ge 0.
 \]
Fix $\mu\in \cM_{\al,A}\cap \cL_{\beta,B}$ and set
$$\xi_{M, \ell}:=\sup_{u\in [\ell-1,\ell+1]} \sup_{s\ge 0}\left| \frac {\P \left(\inf_{t\in [s,s+M]}X(t)>u \right) } {\P\left(\inf_{t\in [0,M]}f_\mu(t)>u\right)}-1\right|.$$
 {Then there exist $C=C(\eta,c)$ and $T_0=T_0(M,\eta,c)$, such that, for all $T\ge \max\Big(T_0,\frac{1}{A}\Big)$},
\begin{align*}
\theta_{\mu}^\ell(M)-\widetilde{\theta}^\ell(T)\le C M^{\frac{1-\eta}{2+\eta}}+\frac { \xi_{M,\ell}} M,
\end{align*}
where $\widetilde{\theta}^\ell(T):=-\frac{1}{T}\log \P(\inf_{t\in [0,T]}X(t)>\ell).$
\end{lem}

\begin{proof}[Lemma~\ref{lem:sampling_fast} implies Proposition~\ref{prop: summable}]
Given $\rho$ as Proposition~\ref{prop: summable}, there exists $c=c(\rho)$ such that
\[
\widehat{\rho}(t)  \le \frac{c}{|t|^2}, \forall t\in\R.
\]
Then, setting $X(\cdot) = f_\rho(\cdot)$, we satisfy the conditions of Lemma~\ref{lem:sampling_fast} with
$\xi_{M,\ell}=0$ and $\eta=2$.  The lemma thus yields the existence of $C(\rho), T_0(M,\rho)<\infty$ such that, for all  {$T\ge T_0(M,\rho)$} we have
\[
\theta^\ell_\rho(M) \le \theta^\ell_\rho(T)  + CM^{-1/4}.
\]
Upon taking $\limsup_{M\to\infty}\liminf_{T\to\infty}$ we arrive at
\[
\limsup_{M\to\infty}\theta^\ell_\rho(M) \le \liminf_{T\to\infty} \theta^\ell_\rho(T),
\]
and the existence of $\lim_{T\to\infty} \theta^\ell_\rho(T) $ follows.
The limit must be positive by Lemma~\ref{lem: FF UB}.
\end{proof}

\begin{proof}[Proof of Lemma~\ref{lem:sampling_fast}]
We extend the methods of Dembo-Mukherjee appearing in \cite[Theorem 1.6]{DM} (see also \cite[Lemma 3.1]{DM2}) to the case where correlations are not necessarily non-negative.

Let $\delta=C' M^{\frac{1-\eta}{2+\eta}}$ where $C'=C'(\eta)>0$ will be chosen later.
Let $T\ge M$ be given.

For $i\ge 1$, set $s_i:=(1+\delta)Mi$, $I_i:=[s_i-M,s_i]$ and $N:=\lfloor\frac{T}{M(1+\delta)}\rfloor$. Note that $\{I_i\}$ are disjoint and
$\cup_{i=1}^N I_i \subset [0,T]$. Consequently,
\[
\P\left(\inf_{[0,T]}X(t)>\ell\right)\le \P\left(\inf_{t\in \cup_{i=1}^N I_i} X(t)>\ell\right).
\]
Define an $N\times N$ matrix ${\bf B}_N$ by setting ${\bf B}_N(i,i):=1$ and, for $i\neq j$,
\[
{\bf B}_N(i,j):=\frac{c}{\gamma}\sup_{s\in I_i,t\in I_j}|s-t|^{-\eta}
\le \frac{c|i-j|^{-\eta}}{\gamma  M^\eta\delta^\eta},
\]
where $\gamma:=\gamma_{M,\delta}=4c (\delta M)^{-\eta}\sum_{i=1}^\infty i^{-\eta}$, so that
$$\max_{1\le i\le N}\sum_{j\ne i}{\bf B}_N(i,j)\le \frac{2c}{\gamma \delta^\eta M^{\eta}}\sum_{i=1}^\infty i^{-\eta}\le \frac{1}{2}.$$
Thus, by the Gershgorin circle theorem, all the eigenvalues of ${\bf B}_N$ lie within the interval $\left[\tfrac 1 2, \tfrac 3 2\right]$, and hence ${\bf B}_N$ is positive definite.
Setting $r(s,t):=\E X(s)X(t)$, we claim that for any $s\in I_i, t\in I_j$ we have
\begin{align}\label{eq:slep}
r(s,t)\le (1-\gamma)r(s,t)\ind_{i=j}+\gamma {\bf B}_N(i,j).
\end{align}
To see this, observe that \eqref{eq:slep} is equivalent to
\begin{eqnarray*}
r(s,t)\le 
\begin{cases}
(1-\gamma)r(s,t)+\gamma, &i=j,\\
c \sup_{s\in I_i,t\in I_j}|s-t|^{-\eta}, &i\ne j,
\end{cases}
\end{eqnarray*}
both of which are immediate from our assumptions. The RHS of \eqref{eq:slep} is the correlation function of the centered non-stationary Gaussian process on $\cup_{i=1}^NI_i$ defined by $t\mapsto \sqrt{1-\gamma}X^{(i)}(t)+\sqrt{\gamma}Z_{i}$ for $t\in I_i$, where
\begin{itemize}
\item ${\bf Z}:=(Z_1,\cdots,Z_N)$ is a centered Gaussian vector with covariance ${\bf B}_N$,
\item $\left\{X^{(i)}\right\}$ are i.i.d.~copies of $X(\cdot)$, independent of ${\bf Z}$.
\end{itemize}
 Using Slepian's inequality (Proposition \ref{prop: slep}) together with \eqref{eq:slep} yields
\begin{align}\label{eq:slep_2}
\notag \P\left(\inf_{t\in \cup_{i=1}^N I_i}X(t)>\ell\right)
&\le \P\left(\inf_{t\in I_i}\sqrt{1-\gamma}X^{(i)}(t)+\sqrt{\gamma}Z_i>\ell,1\le i\le N\right)\\
\notag&= \E\left[ \prod_{i=1}^N \P\left(\inf_{t\in I_i} X(t)>\ell-\frac{\sqrt{\gamma}}{\sqrt{1-\gamma}}Z_i \, \Big| \,\, {\bf Z}\right)\right]\\
\notag&\le \E \prod_{i=1}^N \left[ \P\left(\inf_{t\in I_i}X(t)>\ell-2\ep\right)+\ind\left\{Z_i>\ep \gamma^{-1/2}\right\}\right]\\
&\le \E \prod_{i=1}^N \left[ \P\left(\inf_{t\in [0,M]}f_\mu(t)>\ell-2\ep\right)\left(1+\xi_{M,\ell}\right)+\ind\left\{Z_i>\ep \gamma^{-1/2}\right\}\right],
\end{align}
for any $\ep\in (0,\tfrac 1 2)$, where the one before last line uses the fact that $\sqrt{1-\gamma}\ge \frac{1}{2}$ for $\delta M$ large enough (depending on $c,\eta$).
Next we note that for any collection of distinct indices $i_1,\cdots,i_m\in [N]$,
the covariance matrix $\Sigma$ of $(Z_{i_1},\dots, Z_{i_m})$ has eignevalues within $[\tfrac 1 2, \tfrac 3 2]$. Consequently,
\begin{align*}
\P\left( Z_{i_\ell}>\ep\gamma^{-1/2}, \, 1\le \ell\le m \right)
&=\det(\Sigma)^{-1/2}(2\pi)^{m/2}\int_{(\ep \gamma^{-1/2},\infty)^m}e^{-\tfrac 1 2 {\bf z}^T\Sigma^{-1}{\bf z}}d{\bf z}\\
& \le \frac{2^{m/2}}{(2\pi)^{m/2}} \int_{(\ep \gamma^{-1/2},\infty)^m}e^{-\tfrac 1 3 {\lVert\bf z\rVert_2^2}}d{\bf z}\\
&=3^{m/2}\P\left(Z>\sqrt{\tfrac 2 3}\ep \gamma^{-1/2}\right)^m,
\end{align*}
where $Z\sim \cN(0,1)$. Along with \eqref{eq:slep_2}, this gives
\begin{align}\label{eq:bd}
\notag&\P\left(\inf_{t\in [0,T]}X(t)>\ell\right)\\
\notag& {\le \sum_{m=0}^N  \sum_{1\le i_1< i_2,\ldots< i_m\le N}\left[\P\left(\inf_{t\in [0,M]}f_\mu(t)>\ell-2\ep\right)(1+\xi_{M,\ell})\right]^{N-m}\P(Z_{i_\ell}>\ep \gamma^{-1/2}, 1\le \ell\le m)}\\
\notag&\le \sum_{m=0}^N {N\choose m}
\left[\P\left(\inf_{t\in [0,M]}f_\mu(t)>\ell-2\ep\right)\left(1+\xi_{M,\ell}\right)\right]^{N-m} 3^{m/2}\P\left(Z>\sqrt{\tfrac 2 3} \ep \gamma^{-1/2}\right)^m\\
&=\left[ \P\left(\inf_{t\in [0,M]}f_\mu(t)>\ell-2\ep\right)\left(1+\xi_{M,\ell}\right)+\sqrt{3}\P\left(Z>\sqrt{\tfrac 2 3} \ep \gamma^{-1/2}\right)\right]^N.
\end{align}
Standard bounds on Gaussian tails (Lemma \ref{lem: tail}) along with the definition of $\g$ 
give
\[
\P\left(Z> \sqrt{\tfrac 2 3} \ep \gamma^{-1/2}\right)\le \exp\left(-\frac{\ep^2}{3\gamma}\right)=\exp\left(-C(\eta,c)\ep^2 \delta^\eta M^\eta \right)
\]
for some $C(\eta,c)>0$. By Lemma~\ref{lem: FF LB} there exists $K\in (0,\infty)$ such that, for all $M>1$,
\begin{equation}\label{eq: the famous K}
\P\left(\inf_{t\in [0,M]}f_\mu(t)>\ell-2\ep\right)\ge \P\left(\inf_{t\in [0,M]}f_\mu(t)>\ell-1\right)\ge e^{-KM}.
\end{equation}
Thus, there exist $C'=C'(\eta,c)$ such that for the choice $\varepsilon=\delta=C' M^{\frac{1-\eta}{2+\eta}}$ and all $M>0$ we have
 {$$\P\left( Z \ge \sqrt{\tfrac 2 3} \ep \gamma^{-1/2} \right) \le \P\left(\inf_{t\in [0,M]}f_\mu(t)>\ell-2\ep\right).$$
Plugging this into \eqref{eq:bd} yields
\begin{align*}
\P\left(\inf_{[0,T]}X(t)>\ell\right)\le &\left[\left(1+\xi_{M,\ell}+\sqrt{3}\right) \P\left(\inf_{t\in [0,M]}f_\mu(t)>\ell-2\ep\right)\right]^{N}.
\end{align*}
Taking $\log$ on both sides and dividing by $T$ gives
\begin{align*}
\frac{1}{T}\log \P\left(\inf_{[0,T]}X(t)>\ell\right)
&\le \frac{N}{T}\left[ \log \left(1+\xi_{M,\ell}+\sqrt{3}\right) + \log  \mathcal{P}_\mu^{\ell-2\ep}(M)\right]
\\& \le \frac {\sqrt{3}}{M} + \frac{\xi_{M,\ell}}{M} + \frac 1 {M} \log  \mathcal{P}_\mu^{\ell-2\ep}(M) +
\left| \left(\frac {N}{T}-\frac 1{M}\right)\log \mathcal{P}^{\ell-2\ep}_\mu (M) \right|
\\ & \le \frac{\sqrt{3} + K + \xi_{M,\ell} }{M} + \frac 1 {M} \log  \mathcal{P}_\mu^{\ell-2\ep}(M),
\end{align*}}
where in the last step we used \eqref{eq: the famous K} and the fact that $MN \le T \le M(N+1)$.
Applying Lemma~\ref{lem: level cont}, and noting that $\ep=\delta =C' M^{\frac{1-\eta}{2+\eta}}$, we obtain
\[
\theta_\mu^\ell(M) - \widetilde{\theta}^\ell(T) \le  C'' \delta + \frac {\xi_{M,\ell}}{M},
\]
 {for $T>\max\Big(T_0(M,\eta,c), \frac{1}{A}\Big)$.}
The proof is concluded on recalling the definition of $\delta$.
\end{proof}

\subsubsection{Proof of Proposition~\ref{prop: equal in}}\label{sec: equal in}
Let $\ep>0$ and $T>0$.  {Let $c$ be the constant whose existence is guaranteed by Lemma~\ref{lem: meas cont}-\eqref{item: cont b}, and denote $L=\frac c{\sqrt{\ep}}$.} Let $\mu, \nu \in \mathcal{M}_{(0,\al'),\Alt}\cap \mathcal{M}_{\alpha,A}\cap \mathcal{L}_{\beta,B}$ be such that
$d_{TV}(\mu,\nu)<\varepsilon$ and $\mu|_{[-\frac{L}{T},\frac{L}{T}]}=\nu|_{[-\frac{L}{T},\frac{L}{T}]}$. 
By Claim~\ref{clm: HJ}, there exists a measure $\g$ such that
\begin{itemize}
\item $\g\ge \max\left\{\mu,\nu\right\}$,
\item $\max\left\{d_{\TV}(\g,\mu), d_{\TV}(\g,\nu) \right\} < 2\ep$, and
\item $\g = \mu = \nu$ on $[-\tfrac L T, \tfrac L T]$.
\end{itemize}
Using Lemma \ref{lem: meas cont}-\eqref{item: cont a}, if  {$T\ge \max\{4,\frac 1 A, \frac 1 {\ep}\}$} then
\begin{align}\label{eq:monotone_3}
 \ttt{\g}\le \ttt{\mu}+C^{(1)}_\ep,
\end{align}
where $\lim\limits_{\ep\to 0}C^{(1)}_\ep=0$. Then, invoking Lemma~\ref{lem: meas cont}-\eqref{item: cont b}  {(used with $A=\tfrac L T$, which inherently satisfies $T=\frac{c}{A\sqrt{\ep}}$)}, 
we have 
\begin{align}\label{eq:monotone_4}
\tttlt{\nu}{\ell}{ T(1-c\ep^{1/8}) } \le  \ttt{\g}  + C^{(2)}_{\ep},
\end{align}
where $C_{\ep}^{(2)}$ depends on $\alt$ and $\lim\limits_{\ep\to 0} C^{(2)}_\ep=0$.
Combining \eqref{eq:monotone_3} and \eqref{eq:monotone_4} yields
$$ \tttlt{\nu}{\ell}{ T(1-c\ep^{1/8}) } \le \ttt{\mu}+ C^{(1)}_\ep + C^{(2)}_\ep, $$
provided that  {$T\ge \max\{4,\frac 1 A, \frac {1}{\ep}\}$}, as desired.

\subsubsection{Proof of Proposition~\ref{prop: equal out}}\label{sec: equal out}


Let $\rho= \mu|_{\R\setminus [-\frac L T,\frac L T] } = \nu|_{\R\setminus [-\frac L T,\frac L T] }$ and fix $\tau=\tau(L,\ep,m) >0$ to be chosen later.
Write $n = \lceil \frac{L}{\tau}\rceil$ and let $\{I_j\}_{j\in\pm[n]}$ be the  decomposition of $[-\frac L T,\frac L T]$ used in Lemma~\ref{lem: interval decomp}.
For $\chi\in \{\mu,\nu\}$,
we write $\chi^{j} := \chi(I_j\cup I_{-j})$ and
\[
U^\chi_n (t) = \sum_{j=1}^{ \enn  }    \sqrt{\chi^j} \Big(\zeta^\chi_j C^\chi_j(t)\oplus \eta^\chi_j S^\chi_j(t)\Big).
\]
Using Lemma~\ref{lem: interval decomp}, we
obtain the decomposition
\begin{align*}
f_\chi(t) &=
 U^\chi_{\enn}(t) \oplus R^\chi_{\enn}(t) \oplus f_\rho(t),
\end{align*}
where $R_n^\chi$ is a Gaussian process, which satisfies for all  {$T\ge L m^{1/3}$},
\begin{gather}
\begin{aligned}
&\var\big(R^\chi_\enn(t)\big) \le \tfrac 1 2 \tau^2 \chi\big([-\tfrac L T, \tfrac L T]\big) \le \tfrac 1 2 \tau^2 m \tfrac L T,\\
&\var\big(R^\chi_{\enn}(t)-R^\chi_{\enn}(t+h)  \big) \le (\tfrac L T)^2 \chi\big([-\tfrac L T,\tfrac L T]\big)h^2 \le m (\tfrac L T)^3 h^2\le h^2.
\end{aligned}\label{eq: 2.23 conds ver}
\end{gather}
Next, we couple $f_\mu$ and $f_\nu$ by taking $\zeta^\mu_j=\zeta^\nu_j$ and  $\eta^\mu_j=\eta^\nu_j$ with $(U^\mu_n(t),R^\mu_n(t))$ and $(U^\nu_n(t),R^\nu_n(t))$ independent.
Observe that by containment of events, we have for all $\delta>0$,
\begin{equation*}
\per{\mu} \le \perl{\nu}{\ell-3\delta} + \P\left(\sup_{[0,T]} |U_n^\mu-U_n^\nu| > \delta \right)  +
 \P\left(\sup_{[0,T]}|R_n^\mu| > \delta \right)   + \P\left(\sup_{[0,T]}|R_n^\nu| > \delta\right).
\end{equation*}

 {In the remainder of the proof, we show that there exists constants $c_1,c_2$ such that for the choice $\delta^2 =c_1 (L^5 \ep^2 m)^{1/3}$ (which ensures $\delta\to 0$ as $\ep\to 0$)} we have
\begin{equation}\label{eq: 5.3 main}
\frac{\per{\mu}}2>\P\left(\sup_{[0,T]} |U_n^\mu-U_n^\nu| > \delta \right)  +
 \P\left(\sup_{[0,T]}|R_n^\mu| > \delta \right)   + \P\left(\sup_{[0,T]}|R_n^\nu|> \delta\right),
\end{equation}
for all  $T\ge \max\left\{c_2\Big(\frac{1}{\delta}+1\Big),Lm^{1/3},L\right\}$.
Indeed, this will imply that
$\ttt{\mu} \le \tttl{\nu}{\ell-3\delta}+\tfrac{\log2}T$, from which, using continuity of $\ttt{\mu}$ as a function of $\ell$ (Lemma~\ref{lem: level cont}), we will obtain
$\ttt{\mu} \le \ttt{\nu}+C\delta$  {for the same range of $T$ (by possibly increasing $c_2$), for some $C=C(\alpha, \beta, B,\ell).$}
As $\mu$ and $\nu$ are interchangeable, the proposition will readily follow.

We first bound $\P\left(\sup_{[0,T]} |R_n^\chi| > \delta \right)$ for $\chi\in\{\mu,\nu\}$.  By \eqref{eq: 2.23 conds ver} we may apply Lemma~\ref{lem: win ball}  {to obtain the existence of a universal constant $c$ such that for $T\ge c_3\Big(\frac{1}{\delta}+1\Big)$} we have
\begin{equation}\label{eq: R>eta}
\P\left(\sup_{[0,T]} |R_n^\chi| > \delta \right)
 \le
2 e^{-\frac {\delta^2}{4\tau^2 m L } T}.
\end{equation}

Next we bound $\P\left(\sup_{[0,T]} |U_n^\mu-U_n^\nu| > \delta \right)$. Once again we wish to apply Lemma~\ref{lem: win ball}; let us verify its conditions.
We first compute
\begin{equation}\label{eq: 5.3 goal 2}
\var\left( U_n^\mu(t) -U_n^\nu(t) \right)
= \sum_{j=1}^n \left( \sqrt{\mu^j }  C_j^{\mu}(t) - \sqrt{\nu^j } C_j^\nu(t)  \right)^2
+\sum_{j=1}^n\left( \sqrt{\mu^j }  S_j^{\mu}(t) - \sqrt{\nu^j } S_j^\nu(t)  \right)^2.
\end{equation}
We bound the first term by
\begin{equation}\label{eq: to bound}
\sum_{j=1}^n \left( \sqrt{\mu^j }  C_j^{\mu}(t) - \sqrt{\nu^j } C_j^\nu(t)  \right)^2\le 2\sum_{j=1}^n \mu^j \left( C_j^{\mu}(t) - C_j^\nu(t)  \right)^2
+ 2 \sum_{j=1}^n \left(\sqrt{\mu^j}-\sqrt{\nu^j} \right)^2  C_j^\nu(t)^2.
\end{equation}
We note that
\[
|C^\mu_j(t) - \cos( j \tfrac{\tau}{T} t) |  = \int_{I_j} \left( \cos(\lm t) - \cos(j\tfrac{\tau}{T}t) \right) \frac{d\mu(\lm)}{\mu(I_j)}
\le \tau,
\]
so that
\begin{equation}\label{eq: 5.19 first}
\sum_{j=1}^n \mu^j \left( C_j^{\mu}(t) - C_j^\nu(t)  \right)^2 \le 4 \tau^2 \sum_{j=1}^\enn \mu^j \le 4 \tau^2 m \frac{L}{T}.
\end{equation}
Next we notice that, if $\max\{\mu^j,\nu^j\} \le \frac{\tau \ep \enn}{T}$, then
\[
\left(\sqrt{\mu^j}-\sqrt{\nu^j} \right)^2 \le \frac{\tau \ep \enn}{T};
\]
while if $\max\{\mu^j,\nu^j\} > \frac{\tau \ep \enn}{T}$ then
\[
\left(\sqrt{\mu^j}-\sqrt{\nu^j} \right)^2
=\frac{\left( \mu^j - \nu^j \right)^2 }{\left(\sqrt{\mu^j}+\sqrt{\nu^j}\right)^2 }  \le \frac{T}{\tau \ep \enn}\left(\mu^j-\nu^j \right)^2.
\]
Using the assumption that $|\mu(-\lm,\lm) - \nu(-\lm,\lm)|\le 2 \ep \lm$ for all $\lm\in(0, {\frac{L}{T}})$, we have
\[
\left|\mu^j-\nu^j\right| \le  \left| \mu\left( -\tfrac{j\tau}{T},\tfrac{j\tau}T\right) -  \nu\left( -\tfrac{j\tau}{T},\tfrac{j\tau}T\right) \right|
+ \left| \mu\left( -\tfrac{(j-1)\tau}{T},\tfrac{(j-1)\tau}T \right) -  \nu\left( -\tfrac{(j-1)\tau}{T},\tfrac{(j-1)\tau}T\right) \right|
\le 4 \frac{j \tau \ep}{T}.
\]
Recalling that $C_j(t)\le 1$, we obtain
\[
\sum_{j=1}^n \left(\sqrt{\mu^j}-\sqrt{\nu^j} \right)^2  C_j^\nu(t)^2
\le \enn \left(\frac{\tau \ep \enn}{T} + \frac{T}{\tau \ep \enn} \cdot \left(4 \frac{\enn \tau \ep}{T}\right)^2\right)=
17\enn \cdot \frac{\tau \ep \enn}{T}\le 2^7 \cdot \frac{\ep L^2}{\tau T},
\]
where the last step uses $\enn = \lceil \frac {L}{\tau}\rceil \le 2\frac{L}{\tau}$.
Putting this together with \eqref{eq: 5.19 first} into \eqref{eq: to bound}, we get
\[
\sum_{j=1}^n \left( \sqrt{\mu^j }  C_j^{\mu}(t) - \sqrt{\nu^j } C_j^\nu(t)  \right)^2
\le 2^8\frac{L}{T} \left(\tau^2m + \ep \frac{L}{\tau}\right).
\]
Applying the same chain of arguments yields $\sum_{j=1}^n\left( \sqrt{\mu^j }  S_j^{\mu}(t) - \sqrt{\nu^j } S_j^\nu(t)  \right)^2 \le 2^8 \frac{L}{T} \left(\tau^2m + \ep \frac{L}{\tau}\right)$.
Plugging these bounds into \eqref{eq: 5.3 goal 2}, we conclude that,
\[
\var\left( U_n^\mu(t) -U_n^\nu(t) \right)  \le 2^8  \frac{L}{T} \left(\tau^2m + \ep \frac{L}{\tau}\right).
\]
Next, writing $\Delta_h U^\chi_n(t):=U_n^\chi(t+h)-U_n^\chi(t)$, and noting that $\mu=\nu$ on $[-\frac{L}{T},\frac{L}{T}]$, for all $T\ge L$ we have

\begin{align*}
\var \big( \Delta_h U^\mu_n(t) -\Delta_h U^\nu_n(t) \big)
 &
=\var \big( \Delta_h U^{\mu_1}_n(t) -\Delta_h U^{\nu_1}_n(t) \big)
\\
 &\le 2 \var ( \Delta_h U^{\mu_1}_n(t) ) +2 \var ( \Delta_h U^{\nu_1}_n(t) )\\
&\le    2 \var ( f_{\mu_1}(t+h) -f_{\mu_1}(t) ) +2 \var ( f_{\nu_1}(t+h) -f_{\nu_1}(t) )
\\ & \le  {2}\left( \mu([-1,1]) + \nu([-1,1]) \right)h^2
\end{align*}
where the last inequality uses Obs.~\ref{obs: comp bound on r}. Since $\mu,\nu\in \cL_{\beta,B}$, the quantities $\mu([-1,1])$ and  $\nu([-1,1])$ have an upper bound depending only on $\beta$ and $B$. We have thus shown the existence of  {$a=a(\beta,B)$} for which
\[
\var \big( \Delta_h U^\mu_n(t) -\Delta_h U^\nu_n(t) \big) \le a^2 h^2.
\]
 {Applying Lemma~\ref{lem: win ball}, for $T\ge c_4\Big(\frac{1}{\delta}+1\Big)$ we have
\begin{equation}\label{eq: n1}
\P(\sup_{[0,T]} |U_n^\mu-U_n^\nu| > \delta ) \le 2 \exp\left(-  \frac{\delta^2 }{2^{11} L (\tau^2 m + \ep L /\tau) }T\right),
\end{equation}
where $c_4=c_4(\beta,B)$.}

We now select $\tau =\left( \frac{\ep L }{m}\right)^{1/3}$ and obtain

\begin{equation}\label{eq: n11}
\P(\sup_{[0,T]} |U_n^\mu-U_n^\nu| > \delta ) \le  2 \exp\left(-  \frac{\delta^2 }{2^{12} (m L^ {5} \ep^2)^{1/3} } T\right).
\end{equation}

Applying Lemma~\ref{lem: FF LB}, there exists $R\in(0,\infty)$ such that  $\per{\mu}>e^{-R T}$ for all $T\ge 1$. Using \eqref{eq: R>eta} and \eqref{eq: n11}, we reduce
\eqref{eq: 5.3 main} into showing that
\[e^{-R T}\ge 2 e^{-  \frac{\delta^2 }{2^{12} ( L^ {5}\ep^2m)^{1/3} } T} + 4 e^{-\frac{\delta^2 }{4 (L^{ {5}} \ep^2 m)^{1/3} } T}.\]
Setting $\delta^2=2^{13}R (L^5 \ep^2 m)^{1/3}$, this indeed holds for  {$T\ge \frac{\log 6}{R}$}. Consequently, there exists $c_2=c_2(\alpha,\beta,B,\ell)$ such that\eqref{eq: 5.3 main} holds for $T\ge \max\left\{c_2\Big(\frac{1}{\delta}+1\Big), Lm^{1/3},L\right\}$, as desired.

\subsection{An example of non-existence}\label{sec: exist counter}
In this section we prove Remark~\ref{rmk: cexm}. More precisely, we show that, for some values of  $0<a<b<\infty$, the absolutely continuous measure $\rho=\rho_{a,b}$ whose density is
\[
w_{a,b}(\lm) = \left(\frac{b+a}{2} + \frac{b-a}{2}\cos\left(\tfrac 1 \lm\right)  \right)\textrm{\ind}_{|\lm|\le 1}
\]
does not admit existence of the $0$-level persistence exponent $\theta_\rho^0$.

\medskip
Clearly $a\le w_{a,b}(\lm)\le b$ for all $|\lm|\le 1$. We observe that
\begin{equation}\label{obs: a-b}
\liminf\limits_{\lm\to 0} \frac{\rho_{a,b}([0,\lm])}{\lm} = a,   \quad \limsup\limits_{\lm \to 0} \frac{ \rho_{a,b}([0,\lm]) }{\lm} =b.
\end{equation}

Next, we make use of results from \cite{FFN}, simplifying many constants due to our specific form of spectral density.
For instance, due to the fact that $d\rho(\lm) \ge a \ind_{[-1,1]}(\lm) d\lm$, we may take $|E| = 2$ and $\nu=a$ in the notations of that paper.
Applying \cite[Thm. 5.1 and Rmk. 2]{FFN} with the appropriate parameters, namely
\[
\gamma=1 \: (\text{implying } k=0, s=1, r=\tfrac 1 2, \theta=1), |E|=2, \nu=a,  q=1,
\]
 we have, for all fixed  $\ell>0$ and all $T>T_0(\ell)$, that
\begin{equation}\label{eq: near a}
\rho\left(\left[0,\tfrac 1 T\right]\right) \le  \frac{a+\ep}{T} \quad \Rightarrow  \quad
\perl{\rho}{0} \le 2\P\left(\sqrt{ a+\ep }\,  Z > \ell \sqrt{T}    \right) + 2  \P\left( c_1 \sqrt{1- \tfrac {\ep}{a+\ep} } \, |Z|< \ell \right)^{c_2 T},
\end{equation}
where 
$c_1$ and $c_2$ are universal constants and $Z\sim\cN (0,1)$ is a standard normal random variable.
Next by \cite[Theorem 4.1]{FFN}
we have, for all $\ell>\ell_0(b)$ and all $T>0$, that
\begin{equation}\label{eq: near b}
\rho\left(\left[0,\tfrac 1 T\right]\right) \ge  \frac{b-\ep}{T} \quad \Rightarrow  \quad
\perl{\rho}{0} \ge \P\left(\sqrt{ b-\ep }\,  Z > \ell \sqrt{T}    \right) \cdot \P\Big( \beta(b)\  |Z|< \ell \Big)^{T},
\end{equation}
where $\ell_0(b)$ and $\beta(b)$ are constants which depends only on $b$.

We proceed by applying the above with $\ep=\tfrac a 2$.
For a given $b>0$, fix $\ell_1>\ell_0(b)$.
By Lemma~\ref{lem: tail}, there exists $\theta_1>0$ such that
\begin{equation*}
 \P\left(\sqrt{ b/2}\,  Z > \ell_1 \sqrt{T}    \right) \cdot \P\Big( \beta(b)\  |Z|< \ell_1 \Big)^{T}
 \ge e^{-\theta_1 T},
\end{equation*}
for all large enough $T$.  Now fix $\theta_2>\theta_1$,
and choose $\ell_2>0$ such that
\begin{equation*}
 2  \P\left( c_1 \sqrt{\tfrac 2 3 } \, |Z|< \ell_2 \right)^{c_2 T} \le \tfrac 1 2 e^{-\theta_2 T}.
\end{equation*}
Again by Lemma~\ref{lem: tail} we may choose $a\in (0,\frac b 3)$ so small that
\[
2\P\left(\sqrt{ \tfrac{3a}{2} }\,  Z > \ell_2 \sqrt{T}    \right)  \le \tfrac 1 2 e^{-\theta_2 T}.
\]
Combining these choices with \eqref{eq: near a} and \eqref{eq: near b}, we obtain that
\begin{align*}
& \rho\left(\left[0,\tfrac 1 T\right]\right) \ge \frac{b}{2 T}   \quad \Rightarrow  \quad \perl{\rho}{0} \ge e^{-\theta_1 T},
\\
& \rho\left(\left[0,\tfrac 1 T\right]\right) \le \frac{3a}{2T}   \quad \Rightarrow  \quad \perl{\rho}{0} \le e^{-\theta_2 T},
\end{align*}
for all large enough $T$.
Recalling~\eqref{obs: a-b} and the definition of $\theta^0_\rho(T)$ in \eqref{eq: per}, we conclude that
\[
\liminf_{T\to\infty} \theta_\rho^0(T) \le \theta_1 < \theta_2 \le \limsup_{T\to\infty} \theta_\rho^0(T),
\]
which implies that the persistence exponent $\theta_\rho^0$ does not exist.


\section{Monotonicity of the ball and persistence exponents}\label{sec: s thm}

\subsection{Proof of Theorem~\ref{thm: singular}--\eqref{item: singular ball}}
Let $\beta,B>0$ be such that $\rho, \nu\in \cL_{\beta,B}$.
We may assume that $\ell>0$, since otherwise $\psi^\ell_{\rho+\nu} = \psi^\ell_\rho = \infty$.
By Lemma~\ref{lem: spec decomp} we have $f_{\rho+\nu} \overset{d}{=} f_\rho \oplus f_\nu$.
An application of Anderson's inequality (Proposition~\ref{prop: and}) gives
\[\psi_\rho^\ell \le \psi_{\rho+\nu}^\ell.\]

Next assume that $\nu$ is purely-singular. Using Lemma~\ref{lem: ball-level}(a) we get
$
\psi_{\rho+\nu}^\ell \le \psi_\rho^{\ell-\delta} + \psi_\nu^\delta.
$
Using
Theorem~\ref{thm: nontrivial ball},
we obtain that $\psi_\nu^\delta=0$. By Theorem~\ref{thm: cont}--\eqref{item: cont ball} we have $\psi_\rho^{\ell-\delta} \le \psi_\rho^{\ell} + C_\delta$,
where $\lim\limits_{\delta\to 0} C_\delta=0$.
We conclude that
\[
\psi_\rho^\ell \le \psi_{\rho+\nu}^\ell \le \psi_\rho^\ell +C_\delta,
\]
where upon taking $\delta\to 0$ yields $\psi_{\rho+\nu}^\ell = \psi_\rho^\ell$.

\subsection{Proof of Theorem~\ref{thm: singular}--\eqref{item: singular pers}}

Let $\al,A,\beta,B>0$ be such that $\rho\in\cM\cap\cM_{\al,A}\cap \cL_{\beta,B}$ and $\nu\in \cL_{\beta,B}$. Assume that $\nu'(0)=0$. In particular, there exists $\al'>0$ (depending on $\nu$) for which $\nu\in \cM_{(0,\al'),A}$.

Fix $\ep>0$.
 {Using Proposition~\ref{prop: truncate} with $\eta=\ep$, $L=\ep T, D=A$}, there exists  {$\ep_0(\al',A)$ and $T_0(\ep,A)$} such that if $\ep<\ep_0$ and $T\ge T_0$ we have
\[ {(1- {\tfrac{2}{(\ep T)^{1/4}}} )  }
\theta_{\rho+\nu_\ep}^\ell(T(1- {\tfrac{2}{(\ep T)^{1/4}}} ))\le \theta_{\rho+\nu}^\ell(T)+C_{ {\ep}},
\]
where $\lim_{\ep\to 0}C_\ep=0$.  {Keeping $\ep<\ep_0$ fixed and letting $T\to\infty$, we have,} by Theorem \ref{thm: main exist},   
\begin{equation}\label{eq: moment1}\theta_{\rho+\nu_\ep}^\ell\le \theta_{\rho+\nu}^\ell+C_ {\ep}.
\end{equation}
Finally note that \[
d_{\TV_0}(\rho+\nu_\ep,\rho)=\nu'(0) + \nu([-\ep,\ep]) \to 0, \quad \text{as  } \ep\to 0,
\]
as $\nu'(0)=0$. Thus, invoking Theorem \ref{thm: cont} we get $\lim\limits_{\ep\to 0}\theta_{\rho+\nu_\ep}^\ell= \theta_{\rho}^\ell.$  {Taking $\ep\to 0$ in \eqref{eq: moment1}}, we obtain
$$\theta_{\rho}^\ell\le \theta_{\rho+\nu}^\ell,$$
thus verifying the inequality of Part~\eqref{item: singular pers}.

Assume now that $\nu$ is purely singular.
Lemma~\ref{lem: ball-level}(b) yields that, for any $\delta>0$,
\begin{equation}\label{eq: t less}
\theta^{\ell}_{\rho+\nu} \le \theta^{\ell+\delta}_{\rho}+\psi^\delta_{\nu}.
\end{equation} 
Since $\nu$ is purely singular,
Theorem~\ref{thm: nontrivial ball} yields that $\psi^\delta_{\nu}=0$; while from Lemma~\ref{lem: level cont} we deduce that
$\lim\limits_{\delta\rightarrow0}\theta^{\ell+\delta}_{\rho}=\theta^\ell_{\rho}$. Using these two facts in \eqref{eq: t less} yields
$\theta^\ell_{\rho+\nu} \le \theta^\ell_{\rho}$, which completes the proof.

\section{Exponents under sampling}\label{sec: sample}
 In this section we prove Theorem \ref{thm: sample} concerning the convergence of the ball and persistence exponents of fine mesh sampling of a continuous-time process, to its continuous-time exponent. This is done in Sections~\ref{subs:convergence1} and \ref{subs:convergence2}. Then, in Section~\ref{subs:non-convergence} we establish the tightness of our criterion by providing an instructive example of non-convergence.

\subsection{Proof of Theorem \ref{thm: sample}--\eqref{item: sample ball}: ball exponent under sampling}\label{subs:convergence1}

Let $\rho\in\cL$ and $\ell>0$. Fix $\D>0$.
Complementing the definition in \eqref{eq: ball}, we set
\[
\psi_{\rho;\D}^\ell(T):=-\frac{1}{T}\log\P\left(\sup_{n\in \Z, \,n\D \in [0,T]} \left|f_\rho(n\D) \right| < \ell\right).
\] 
Define the centered Gaussian (non-stationary) process
$X_\D(t):=f_\rho\Big(\D \lceil \frac{ t}{\D}\rceil \Big)$, and note that $$\psi_{\rho; \D}^\ell(T)=-\frac{1}{T}\log \P\left(\sup_{[0,T]}|X_\D|<\ell\right)\le -\frac{1}{T}\log \P\left(\sup_{[0,T]}|f_\rho|<\ell\right)=\psi_{\rho}^\ell(T).$$
By Corollary~\ref{cor: ball exist} (valid also for discrete processes by Remark~\ref{rmk: Z}), we may take limits as $T\to\infty$ to obtain $\psi_{\rho; \D}^\ell\le \psi_{\rho}^\ell$. It remains to show that
\begin{align}\label{eq:suffice_ball_exp}
\psi_\rho^\ell\le  \liminf_{\D\to 0}\psi_{\rho; \D}^{\ell}.
\end{align}
To this end, fix $\delta>0$ and apply Khatri-Sidak Inequality (Proposition~\ref{prop: KS}) to get
\begin{align*}
\P\left(\sup_{ [0,T]}|f_\rho|<\ell+\delta\right) & \ge \P\left(\sup_{ [0,T]}|X_\D|<\ell,\sup_{ [0,T]}|X_\D-f_\rho|<\delta\right)\\
&\ge  \P\left(\sup_{ [0,T]}|X_\D|<\ell\right)\P\left(\sup_{[0,1]}|X_\D-f_\rho|<\delta\right)^{\lceil T\rceil}.
\end{align*}
 {Using the definition of $X_\Delta$, the first term can be bounded from below as follows:
\begin{align*}
\P\left(\sup_{[0,T]}|X_\Delta|<\ell\right) &=
\P\left(\sup_{n\in \mathbb{Z}, \, 0\le n\le \lceil \frac{T}{\Delta}\rceil}|f_\rho(n\Delta)|<\ell\right)
\\  &\ge \P\left(\sup_{n\in \mathbb{Z}, \, 0\le n\Delta \le  T+\Delta}|f_\rho(n\Delta)|<\ell\right)=e^{-(T+\Delta)\psi^\ell_{\rho;\Delta}(T+\Delta)}.
\end{align*}
}
Consequently, upon taking $\log$, dividing by $T$ and letting $T\to\infty$ we have
\[
\psi_{\rho}^{\ell+\delta} \le  \psi_{\rho; \D}^{\ell}-\log  \P\Big(\sup_{ [0,1]}|X_\D-f_\rho|<\delta\Big).
\]
The sample path continuity of $f_\rho(.)$ yields that $\sup_{ [0,1]}|X_\D-f_\rho|\stackrel{a.s.}{\rightarrow}0$ as $\D\to 0$, and so
\[
\psi_{\rho}^{\ell+\delta} \le  \liminf_{\D\to 0}\psi_{\rho; \D}^{\ell}.
\]
Finally, letting $\delta\to 0$ and using Theorem \ref{thm: cont}--\eqref{item: cont a} gives \eqref{eq:suffice_ball_exp}, as required.

\subsection{Proof of Theorem \ref{thm: sample}--\eqref{item: sample above}: persistence exponent under sampling}\label{subs:convergence2}

For $\ell\in\R$, $\D>0$ and a spectral measure $\rho$, we complement the definitions in \eqref{eq: per} by setting 
\begin{align}\label{eq: per disc}
 \per{\rho; \D} :=\P\left(\inf_{n\in \Z, \, n\D \in [0,T]}f_\rho(n\D )>\ell\right), \qquad \theta_{\rho;\D}^{\ell}(T)  := -\frac 1 T \log \perl{\rho; \D}{\ell}.
\end{align}

Theorem \ref{thm: sample}--\eqref{item: sample above} is a consequence of the following.

\begin{prop}\label{prop: sample}
Suppose $\rho\in \cM$ has density $\rho' \in C^2(\R)$ which is compactly supported.
Then
\begin{equation}\label{eq: theta diff}
 \limsup_{\substack{T\to\infty \\ \D\to 0}}\left| \theta_{\rho; \D}^\ell(T) - {\theta_{\rho}^\ell}  \right| = 0.
 \end{equation}
\end{prop}

\begin{proof}[Proposition \ref{prop: sample} yields Theorem \ref{thm: sample}--\eqref{item: sample above}]
Let $\rho\in \mathcal{L} \cap \mathcal{M}$ be compactly supported, and let $\ep>0$.
Write~$\nu$ for the smooth measure with $d_{\TV}(\rho,\nu)<\ep$ from Claim~\ref{clm: smooth approx}. 
By the triangle inequality we have
\begin{equation}\label{eq: triangle theta}
\left|\theta^\ell_\rho - \theta^\ell_{\rho;\D}\right| \le
\left| \theta_{\rho}^\ell - \theta_{\nu}^\ell \right|  + \left| \theta_{\nu}^\ell - \theta_{\nu;\D}^\ell \right|
+\left| \theta_{\nu; \D}^\ell - \theta^\ell_{\rho; \D} \right|
\end{equation}

By Proposition~\ref{prop: theta smooth approx} we have
\begin{equation}\label{eq: less tremendous}
\left|\theta_\rho^\ell -\theta_{\nu}^\ell \right| < C_\ep,
\end{equation}
where 
$\lim_{\ep\to 0}C_\ep = 0$ and 
$C_\ep$
depends on $\rho$.

 By Observation~\ref{obs: sample}, for any $\D>0$, the  spectral measure of the discrete-time process $\{f(t)\}_{t\in\D\Z}$
is $\rho^*_\D(I) = \rho\left( \bigcup_{n\in \Z}  \left\{I+2\pi \tfrac{n}{\D}  \right\}\right)$
(supported in $[-\frac{\pi}{\D},\frac{\pi}{\D}]$, the dual space of $\D \Z$).
  Since $\rho$ and $\nu$ are compactly supported, $\rho^*_\D = \rho$ and $\nu^*_\D =\nu$ for all small enough~$\D$.
 Applying Proposition~\ref{prop: theta smooth approx} to the sequence $\{f_\rho(t)\}_{t\in\D\Z}$ (possible by Remark~\ref{rmk: Z}), we conclude that, for small enough $\Delta$,
\begin{equation}\label{eq: tremendous}
 \left| \theta_{\rho;\D}^\ell -  \theta_{\nu;\D}^\ell \right| <C_\ep.
\end{equation}

By Proposition~\ref{prop: sample} we have 
\begin{equation}\label{eq: not tremendous}
\lim_{\D \to 0} \left| \theta_{\nu}^\ell - \theta_{\nu;\D}^\ell \right|=0.
\end{equation}
Taking $\limsup$ as $\D$ tends to $0$ on \eqref{eq: triangle theta}, and 
plugging in \eqref{eq: less tremendous}, \eqref{eq: tremendous} and \eqref{eq: not tremendous}, yields
\[
\limsup_{\D\to 0}\left|\theta^\ell_\rho- \theta^\ell_{\rho;\D}\right|
< 2C_\ep.
\]
 As $\lim_{\ep\to 0} C_\ep=0$, the proposition follows.
\end{proof}

\begin{proof}[Proof of Proposition \ref{prop: sample}]
The proof is an application of Lemma~\ref{lem:sampling_fast}.
Fixing $\D>0$ define a (non-stationary) Gaussian process $X_\D(\cdot)$ by setting
\begin{align*}
X_\D(t)=&f_\rho(\D \lfloor \tfrac{t}{\D} \rfloor).
\end{align*}
Then fixing $M>0$ we have
\begin{align*}
\P\left(\inf_{t\in [s,s+M]}X_\D(t)>u\right)=&\P\left(\min_{\left\lfloor \tfrac{s}{\D}\right\rfloor \le i\le \left\lfloor \tfrac{s+M}{\D} \right\rfloor }f_\rho(i\D)>u\right)
=\P\left(\min_{0\le i\le \left\lfloor \tfrac{s+M}{\D}\right\rfloor-\left\lfloor \tfrac{s}{\D}\right\rfloor   }f_\rho(i\D)>u\right),
\end{align*}
which converges to $\P\left(\inf_{t\in [0,M]}f_\rho(t)>u\right)$ uniformly in $s\ge 0$ and in $u\in [\ell-1,\ell+1]$ as $\D\to 0$, by stationarity and continuity of sample paths.
Therefore
\[
\xi_{M,\ell}^{\Delta}:=\sup_{u\in [\ell-1,\ell+1]}\sup_{s\ge 0}\left|   \frac{ \P\left(\inf_{t\in [s,s+M]} X_\D(t)>u\right)}{\P\left(\inf_{t\in [0,M]}f_\rho(t)>u\right)} - 1 \right|
\]
satisfies $\lim\limits_{\D\to 0}\xi_{M,\ell}^{\Delta} =0$.
Moreover,  for some $c=c(\rho)>0$ we have $|r(t)| = |\widehat{\rho}(t)| \le \frac{c}{|t|^2}$, since $\rho$ is compactly supported with $C^2$-density.
This implies
\begin{align*}
\left|\E X_\D(s) X_\D(s+t)\right| = \widehat{\rho}\left(\D \lfloor \tfrac{s}{\D} \rfloor -\D \lfloor \tfrac{s+t}{\D} \rfloor  \right)
\le \frac {c'}{|t|^2},
\end{align*}
 {for some constant $c'=c'(\rho)$.  Fixing $M>0$, an application of  Lemma~\ref{lem:sampling_fast} gives the existence of constants $T_0=T_0(M,\rho)$ and $C=C(\rho)$ such that for all $T\ge T_0$ we have
$$\theta_{\rho;\Delta}^\ell (T)\ge \theta_\rho^\ell(M)- \frac{\xi_{M,\ell}^\Delta}{M}-C M^{-1/4}. $$
Letting $T\to\infty$ and $\Delta\to 0$  we get
\begin{align*}
    \liminf_{\substack{T\to\infty \\ \D\to 0}} \theta_{\rho; \D}^\ell(T)\ge \theta_\rho^\ell(M)-C M^{-1/4},
\end{align*}
which on letting $M\to\infty$ and invoking Theorem \ref{thm: main exist} gives
\begin{align}\label{eq:one_side}
    \liminf_{\substack{T\to\infty \\ \D\to 0}} \theta_{\rho; \D}^\ell(T)\ge \theta_\rho^\ell.
\end{align}
On the other hand, by inclusion of events we trivially have
$$\theta_{\rho;\Delta}^\ell (T)\le \theta_\rho^\ell(T),$$
which on letting $T\to\infty$ gives
\begin{align}\label{eq:other_side}
    \limsup_{\substack{T\to\infty \\ \D\to 0}} \theta_{\rho; \D}^\ell(T)\le \theta_\rho^\ell.
\end{align}
Combining \eqref{eq:one_side} and \eqref{eq:other_side}, the desired convergence in \eqref{eq: theta diff} follows.}
\end{proof}

\subsection{An example of non-convergence under sampling}\label{subs:non-convergence}

\begin{prop}\label{prop:non-convergence}
    There exists an absolutely continuous spectral measure $\rho$, such that
the exponents of the sampled process $\theta^0_{\rho;\frac {1}{k}}$ for $k\in \N$ tend to $0$ as $k\to\infty$, while 
$\theta^0_\rho$ exists, and is strictly positive.
\end{prop}

\begin{proof}
Let $\rho$ be the absolutely continuous measure with density
\[
\rho'(\lm) = \frac{(1-e^{|\lm|}\text{dist}(\lm,2\pi\Z) )_+} {|\lm|} \ind_{|\lm|>\pi} + \ind_{|\lm|<\pi}.
\]

It is clear that $\rho'$ is non-negative and symmetric.
To see that $\rho'\in L^1(\R)$, note that for any $n\in \N$ we have:
$\int_{2\pi n -\pi}^{2\pi n+\pi} \rho'(\lm) d\lm \le \frac {2}{2\pi n-\pi } e^{-2\pi n}$.
It is also clear that $\rho \in \cL\cap \cM$ (in fact, $\rho$ has a finite exponential moment).
Thus Theorem \ref{thm: main exist} implies the existence of
$\theta_\rho^0 \in (0,\infty)$.
Now let us consider the sampled process. By Observation~\ref{obs: sample}, for any $k\in\N$, the discrete-time process $ \{f\left(t\right)\}_{t\in \frac{1}{k}\mathbb{Z}}$
 has the spectral measure $\rho^*_k(I) = \rho\left( \bigcup_{n\in \Z}  \left\{I+2\pi n k \right\}\right)$.
The local density of this measure at $0$ is
 { \[
\liminf_{\ep\to 0 }\frac{ \rho^*_k( (-\ep,\ep))} {2\ep} \ge 
 \sum_{n\in\mathbb{Z}} \liminf_{\ep\to 0} \frac{ \rho( (2\pi nk-\ep,2\pi n k +  \ep))} {2\ep}
= 1+2\sum_{n\in\mathbb{N}} \frac{1}{2\pi n k} = \infty.
 \]}

Applying Proposition~\ref{prop: low explode} to $\rho^*_k$, we obtain that $\theta^0_{\rho; \frac 1 k }
=\theta^0_{\rho^*_k}=0$. As we have seen that $\theta^0_\rho >0$, we conclude that
$\lim\limits_{k\to\infty} \theta^0_{\rho; \frac 1 k} \ne \theta^0_{\rho}$.
\end{proof}

\vspace{15pt}
{\sc{Acknowledgements:}} We acknowledge AIM (the American Institute of Mathematics) for the support and hospitality during a SQuaRE meeting on Persistence probabilities (2017), where this project was initiated. We are grateful to Amir Dembo and Mikhail Sodin for useful discussions and encouragement.  We thank Mikhail Lifshits for simplifying arguments regarding ball probabilities, Liran Rotem for information about log-concavity which led to
Lemma \ref{lem: convolution}, Ori Gurell-Gurevitch for the idea of proof of Proposition \ref{prop: slabs}, Bo'az Klartag for pointing out the reference \cite{BP} and Zemer Kozloff for the reference \cite{Shi}. Finally, we thank an anonymous referee for a careful reading of our draft, and making numerous helpful suggestions which greatly improved the presentation before the paper.


\printbibliography

\end{document}